\documentclass[11pt]{article}
\usepackage[centertags]{amsmath}
\usepackage{amsfonts}
\usepackage{amssymb}
\usepackage{amsthm,color}
\usepackage{newlfont}
\usepackage{bbm}

\newtheorem{Theorem}{Theorem}[section]

\newtheorem{Lemma}[Theorem]{Lemma}

\newtheorem{Remark}[Theorem]{Remark}

\numberwithin{equation}{section}

\begin{document}

\def\le{\left}
\def\r{\right}
\def\cost{\mbox{const}}
\def\a{\alpha}
\def\d{\delta}
\def\ph{\varphi}
\def\e{\epsilon}
\def\la{\lambda}
\def\si{\sigma}
\def\La{\Lambda}
\def\B{{\cal B}}
\def\A{{\mathcal A}}
\def\L{{\mathcal L}}
\def\O{{\mathcal O}}
\def\bO{\overline{{\mathcal O}}}
\def\F{{\mathcal F}}
\def\K{{\mathcal K}}
\def\H{{\mathcal H}}
\def\D{{\mathcal D}}
\def\C{{\mathcal C}}
\def\M{{\mathcal M}}
\def\N{{\mathcal N}}
\def\G{{\mathcal G}}
\def\T{{\mathcal T}}
\def\R{{\mathbb R}}
\def\I{{\mathcal I}}

\def\bw{\overline{W}}
\def\phin{\|\varphi\|_{0}}
\def\s0t{\sup_{t \in [0,T]}}
\def\lt{\lim_{t\rightarrow 0}}
\def\iot{\int_{0}^{t}}
\def\ioi{\int_0^{+\infty}}
\def\ds{\displaystyle}
\def\pag{\vfill\eject}
\def\fine{\par\vfill\supereject\end}
\def\acapo{\hfill\break}

\def\beq{\begin{equation}}
\def\eeq{\end{equation}}
\def\barr{\begin{array}}
\def\earr{\end{array}}
\def\vs{\vspace{.1mm}   \\}
\def\rd{\reals\,^{d}}
\def\rn{\reals\,^{n}}
\def\rr{\reals\,^{r}}
\def\bD{\overline{{\mathcal D}}}
\newcommand{\dimo}{\hfill \break {\bf Proof - }}
\newcommand{\nat}{\mathbb N}
\newcommand{\E}{\mathbb E}
\newcommand{\Pro}{\mathbb P}
\newcommand{\com}{{\scriptstyle \circ}}
\newcommand{\reals}{\mathbb R}

\newcommand{\red}[1]{\textcolor{red}{#1}}

\def\Amu{{A_\mu}}
\def\Qmu{{Q_\mu}}
\def\Smu{{S_\mu}}
\def\H{{\mathcal{H}}}
\def\Im{{\textnormal{Im }}}
\def\Tr{{\textnormal{Tr}}}
\def\E{{\mathbb{E}}}
\def\P{{\mathbb{P}}}
\def\span{{\textnormal{span}}}
\title{Large deviations for the two-dimensional stochastic Navier-Stokes equation with vanishing  noise correlation\thanks{This material is based upon work supported by the National Science Foundation under grant No. 0932078000, while the authors were in residence at the Mathematical Sciences Research Institute in Berkeley, California, during the Fall 2015 semester.}}
\author{Sandra Cerrai\thanks{Partially supported by the NSF grant DMS 1407615.}\\
\normalsize University of Maryland, College Park\\ United States
\and
Arnaud Debussche\thanks{Partially supported by the French government thanks to the
ANR program Stosymap and the ``Investissements d'Avenir" program ANR-11-LABX-0020-01. }\\
\normalsize IRMAR, ENS Rennes, CNRS, UBL, Bruz\\ France}
\date{}

\date{}

\maketitle

\begin{abstract}
We are dealing with the validity of a large deviation principle for the two-dimensional  Navier-Stokes equation, with periodic boundary conditions, perturbed by a Gaussian random forcing. We are here interested in the regime where  both the strength of the noise and its correlation are vanishing, on a length scale $\e$ and $\d(\e)$, respectively, with $0<\e,\ \d(\e)<<1$. Depending on   the relationship between $\e$ and $\d(\e)$ we will prove  the validity of the large deviation principle in different functional spaces.

\end{abstract}

\section{Introduction}
\label{sec1}
In the present paper we are dealing with the following randomly forced two-dimensional incompressible Navier-Stokes equation {with periodic boundary conditions, defined on the domain $D=[0,2\pi]^2$,
\begin{equation}
\label{maineq}
\left\{
\begin{array}{l}
\ds{\partial _t u(t,x)=\Delta u(t,x)-(u(t,x)\cdot\nabla)u(t,x)+\nabla p(t,x)+\sqrt{\e}\,\partial_t\, \xi^{\d}(t,x),\ \ \ \ x \in\,D,\ \ t\geq 0,}\\
\vs
\ds{\text{div}\, u(t,x)=0,\ \ \ \ x \in\,D,\ \ t\geq 0,\ \ \ \ \ u(0,x)=u_0(x),\ \ \ x \in\,D.}
\end{array}\r.
\end{equation}
Here $u$ denotes the velocity and $p$ denotes the pressure of the fluid. Moreover, $\xi^\d(t,x)$ denotes a Gaussian random forcing.
We are here interested in the regime where the noise is weak, that is its typical strength is of order $\sqrt{\e}<<1$,  and almost white in space, that is its correlation decays on a lenght-scale $\d<<1$. 

As well known, in order to have well posedness in $C([0,T];[L^2(D)]^2)$ for equation \eqref{maineq}, the Gaussian noise $\xi^\d$ cannot be white in space. In fact, white noise in space and time has been considered in \cite{DPD02}, where the well-posedness of equation \eqref{maineq} has been studied in some Besov spaces of negative exponent, for $\mu$-almost every initial condition, for a suitable centered Gaussian measure  $\mu$. Here, we assume that for any fixed $\d>0$ the noise $\xi^\d(t,x)$ is sufficiently smooth in the space variable $x \in\,D$ to guarantee that for any initial condition $u_0 \in\,[L^2(D)]^2$ there exists a unique generalized solution in $C([0,T];[L^2(D)]^2)$ (see Section \ref{sec2} for all details).

As a consequence of the contraction principle and of some continuity properties of the solution of equation \eqref{maineq}, it is possible to show  that,  for any  $\d>0$ fixed, the family  $\{{\cal L}(u_{\d,\e})\}_{\e>0}$ given by the laws of the solutions of equation \eqref{maineq} satisfies  a large deviation principle in $C([0,T];L^2(D))$, for any $T>0$ fixed, with rate $\e$ and action functional
\[I_T^\d(f)=\frac 12\int_0^T|Q_\d^{-1}\le(f^\prime(t)-A f(t)-b(f(t))\r)|_H^2\,dt,\]
where $A$ is the Stokes operator, $b$ is the Navier-Stokes nonlinearity and $Q_\d$ is the square root of the covariance of the noise $\xi^\d$ (see Section \ref{sec2} for all definitions and notations and also \cite{BC_2015}).

In \cite{BCF_2013},  the limiting behaviors, as $\d\downarrow 0$, for the large deviation action  functional $I_T^\d$, as well as for the corresponding quasipotential $V^\d$ have been studied. Namely it has been  proven that if the operator $Q_\d$ converges strongly to the identity operator, and a few other conditions are satisfied, then the operators $I_T^\d$ and $V^\d$ converge pointwise, as $\d\downarrow 0$,  to the operator
\begin{equation}
\label{functional-white}
I_T(f)=\frac 12\int_0^T|f^\prime(t)-A f(t)-b(f(t))|_H^2\,dt,\end{equation}
and the operator
\[V(x)=|x|^2_{[H^1(D)]^2},\]
respectively,
where $I_T$ and $V$ would be  the natural candidates for the large deviation action functional in $C([0,T];[L^2(D)]^2)$ and the quasi-potential in $[L^2(D)]^2$, in case equation \eqref{maineq}, perturbed by space-time white noise, were well-posed in $[L^2(D)]^2$. 

Unlike in the present paper, in \cite{BCF_2013} we were interested in the large time limiting behavior of equation \eqref{maineq}, in the case $0<\e<<\d<<1$. Actually, here we are not taking first the limit in $\e$ and then in $\d$, but  we are considering the case in which the parameter $\d$ is a function  of the parameter $\e$ that describes the intensity of the noise, and
\begin{equation}
\label{cond}
\lim_{\e\to 0}\d(\e)=0.\end{equation}
By using the weak convergence approach to large deviations, as described in \cite{BDM} for SPDEs, we show that in this case  the family $\{u_{\e,\d(\e)}\}_{\e>0}$ satisfies a large deviation principle in the space 
$C([0,T];\mathcal{B}^\si_{p}(D))$, where $\mathcal{B}^\si_{p}(D)$ is a suitable Besov space of functions, with $\si<0$ and $p\geq 2$. Moreover, in the case condition \eqref{cond} is integrated with the condition
\[\lim_{\e\to 0}\e\,\d(\e)^{-\eta}=0,\]
for some $\eta>0$, we prove that the family $\{u_{\e,\d(\e)}\}_{\e>0}$ satisfies a large deviation principle in 
$C([0,T];H)$. In both cases, the action functional that describes the large deviation principle is the operator $I_T$ defined in \eqref{functional-white}.

We would like to mention the fact that in \cite{HW} Hairer and Weber have studied a similar problem for the equation
\begin{equation}
\label{hairer}
\left\{
\begin{array}{l}
\ds{\partial _t u(t,x)=\Delta u(t,x)+c\,u(t,\xi)-u^3(t,\xi)+\sqrt{\e}\,\partial_t\, \xi^{\d}(t,x),\ \ \ \ x \in\,D,\ \ t\geq 0,}\\
\vs
\ds{u(0,x)=u_0(x),\ \ \ x \in\,D,}
\end{array}\r.
\end{equation}
where $D$ is a bounded smooth domain either in $\mathbb{R}^2$ or in $\mathbb{R}^3$. By using the recently developed theory of regularity structures, they study the validity of a large deviation principle for the solutions $u_{\e,\d}$ of equation \eqref{hairer}, in the case condition \eqref{cond} is satisfied. Actually, they consider the renormalized equation 
\[\left\{
\begin{array}{l}
\ds{\partial _t u(t,x)=\Delta u(t,x)+(c+3\,\e\,c_{\d(\e)}^{(1)}-9\,\e^2\,c_{\d(\e)}^{(2)})\,u(t,\xi)-u^3(t,\xi)+\sqrt{\e}\,\partial_t\, \xi^{\d}(t,x),}\\
\vs
\ds{u(0,x)=u_0(x),\ \ \ x \in\,D,}
\end{array}\r.\]
where $c_{\d(\e)}^{(1)}$ and $c_{\d(\e)}^{(2)}$ are the constants, depending on the dimension of the underlying space,  arising from the renormalization procedure, and they prove that if \eqref{cond} holds, then the family of solutions $\{u_{\e,\d(\e)}\}_{\e>0}$ satisfies a large deviation principle in $C([0,T],C^\eta(D))$, where $C^\eta(D)$ is some space of functions of negative regularity in space, with action functional 
\[I_T(f)=\frac 12\int_0^T|\partial _t f-\Delta f-c\,f+f^3|_{L^2(D)}^2\,dt.\]
They also study the large deviation principle for equation \eqref{hairer} and prove that if, in addition to \eqref{cond} the following condition holds
\[\lim_{\e\to 0}\e\,\log\d(\e)^{-1}=\la \in\,[0,\infty),\ \ \ \text{for}\ d=2, \ \ \ \ \ \ \ \lim_{\e\to 0}\e\,\d(\e)^{-1}=\la \in\,[0,\infty),\ \ \ \text{for}\ d=3,\] 
then the family $\{u_{\e,\d(\e)}\}_{\e>0}$ satisfies a large deviation principle in $C([0,T],C^\eta(D))$, with respect to the action functional
\[I_T^\la(f)=\frac 12\int_0^T|\partial _t f-\Delta f+c_\la\,f+f^3|_{L^2(D)}^2\,dt,\]
for some explicitly given constant $c_\la$, depending on $\la$ and $d$ and  such that $c_0=-c$.

\section{Notations and preliminaries}
\label{sec2}

We consider here the following incompressible Navier-Stokes  equation with periodic boundary conditions on the two-dimensional domain
$D=[0,2\pi]^2$,
\begin{equation}
\label{eq1}
\left\{
\begin{array}{l}
\ds{\partial _t u(t,x)=\Delta u(t,x)-(u(t,x)\cdot\nabla)u(t,x)+\nabla p(t,x)+\sqrt{\e}\,\partial_t\, \xi^{\d}(t,x),\ \ \ \ x \in\,D,\ \ t\geq 0,}\\
\vs
\ds{\text{div}\, u(t,x)=0,\ \ \ \ x \in\,D,\ \ t\geq 0,\ \ \ \ \ u(0,x)=u_0(x),\ \ \ x \in\,D,}\\
\vs
\ds{u(t,x_1,0)=u(t,x_1,2\pi),\ \ \ u(t,0,x_2)=u(t,2\pi,x_2),\ \ \ (x_1,x_2) \in\,[0,2\pi]^2,\ \ \ t\geq 0,}
\end{array}\r.
\end{equation}
where $0<\e,\d<<1$ are some small positive constants.
Here $\xi^\d(t,x)$ is a Wiener process on $[L^2(D)]^2$, with covariance $Q_\delta$ to be defined below.

We assume that the initial data $u_0$ and the noise $\xi^\d$ are zero average in space. So that $u(t)$ remains with zero average for all time. It is not difficult to get
rid of this assumption.

In what follows, we will denote by $H$ the subspace of $[L^2(D)]^2$ consisting of periodic, divergence free and zero average functions, that is
\[H=\left\{ u \in\,[L^2(D)]^2\,:\,\int_D u(x)\,dx=0,\ \ \text{div}\,u=0,\ \ u\ \text{is periodic in}\ D\,\r\}.\]
 $H$ turns out to be a Hilbert space, endowed with the standard scalar product  $\le<\cdot,\cdot\r>_H$ inherited from $[L^2(D)]^2$.
Moreover, we will denote by $P$ the Leray-Helmholtz projection of $[L^2(D)]^2$ onto $H$. 

Now,  for any $k=(k_1,k_2) \in\,\mathbb{Z}_0^2=\mathbb{Z}^2\setminus \{(0,0)\}$ we define
\[e_k(x)=\frac 1{2\pi}\,\frac{k^\perp}{|k|}\,e^{i\,x\cdot k}=e^{i\,(x_1 k_1+x_2k_2)},\ \ \ x=(x_1,x_2) \in\,D,\ \ \ k \in\,{\mathbb Z}_0,\]
where
\[k^\perp=(k_2,-k_1),\ \ \ \ \ |k|=\sqrt{k_1^2+k_2^2}.\]
It turns out that the family $\{e_k\}_{k \in\,\mathbb{Z}^2_0}$ is a complete orthonormal system in $H_{\mathbb{C}}$, the complexification of the space $H$. For every $s \in\,\mathbb{R}$, we define
\[H^s(D):=\le\{\,u:D\to\mathbb{R}\,:\ |u|_{H^s(D)}^2:=\sum_{k \in\,\mathbb{Z}^2_0}|\langle u,e_k\rangle|^2|k|^{2s}<\infty\,\r\}.\]

Next, for $q \in\,\nat$, we set $\d_q:=\Pi_{2^q}-\Pi_{2^{q-1}}$, where  $\Pi_n$ denote the projection of $H$ into $H_n:=\text{span}\{e_k\}_{|k|\leq n}$. Namely
\[\d_q u=\sum_{2^{q-1}<|k|\leq 2^q} \le<u,e_k\r>_H e_k,\ \ \ \ \ \ u \in\,\bigcup_{s \in\,\reals} H^s(D).\]
For any $\si \in\,\mathbb{R}$ and $p\geq 1$, we define
\[\ds{\mathcal{B}^\si_{p}(D):=\le\{\,u \in\,\bigcup_{s \in\,\reals} H^s(D)\ :\ \sum_{q \in\,\nat} 2^{pq\si}|\d_q u|_{L^p(D)}^p<\infty\r\}.}\]
$\mathcal{B}^\si_{p}(D)$ turns out to be a Banach space, endowed with the norm
\[|u|_{\mathcal{B}^\si_{p}(D)}:=\le(\,\sum_{q \in\,\nat}2^{pq\si}|\d_q u|_{L^p(D)}^p\r)^{\frac 1p}.\]

\bigskip

Now, we define the Stokes operator
\[A u=P \Delta u,\ \ \ \ \ u \in\,D(A)=H\cap [H^2(D)]^2,\]
where $P$ is the Helmodtz projection.
It is immediate to check that for any $k \in\,\mathbb{Z}^2_0$
\[A e_k=-|k|^2 e_k,\ \ \ \ k \in\,\mathbb{Z}^2_0.\]
For any $r \in\,\mathbb{R}$, we denote by $(-A)^r$ the $r$-th fractional power of $-A$, defined on its domain $D((-A)^r)$. It is well known that $D((-A)^r)$ is the closure of the space spanned by $\{e_k\}_{k \in\,\mathbb{Z}^2_0}$ with respect to the norm in  $[H^{2r}(D)]^2$ and the mapping
\[u \in\,D((-A)^r)\mapsto |(-A)^r u|_H \in\,[0,+\infty),\]
defines a norm on $D((-A)^r)$, equivalent to the usual norm in $[H^{2r}(D)]^2$. Moreover, we have that the Leray-Helmholtz projection $P$ maps $[H^{2r}(D)]^2$ into $D((-A)^r)$, for every $r \in\,\mathbb{R}$. 

Due to the incompressibility condition, the nonlinearity in equation \eqref{eq1} can be rewritten as
\[(u\cdot \nabla)v=\text{div}\,(u\otimes v),\]
where
\[u\otimes v=\le(\begin{matrix}
u_1 v_1  &  u_1 v_2\\
u_2 v_1  &  u_2 v_2
\end{matrix}\r).\]
In what follows, we shall set
\begin{equation}
\label{ca147}
b(u,v)=-P\text{div}\,(u\otimes v),\ \ \ \ b(u)=-P\text{div}\,(u\otimes u).\end{equation}

We recall here that, whenever the quantities on the left-hand sides make sense, it holds
\begin{equation}
\label{identities}
\le<b(u),u\r>_H=0,\ \ \ \ \le<b(u),Au\r>_H=0,
\end{equation}
(for a proof see e.g. \cite{Temam_1983}).

Finally, concerning the noisy perturbation $\xi^\d(t,x)$ in equation \eqref{eq1}, it is a  Wiener process  on $[L^2(D)]^2$ and has zero average.
In what follows, we shall set
\[w^\d(t):=P\xi^\d(t),\ \ \ t\geq 0.\]
$w^\d(t)$ is now a  Wiener process on $H$, and we assume it can be written as
\[w^\d(t,x)=\sum_{k \in\,\mathbb{Z}^2_0} \lambda_k(\d) e_k(x) \beta_k(t),\ \ \ \ t\geq 0,\ \ \ x \in\,D,\]
where $\{e_k\}_{k \in\,\mathbb{Z}^2_0}$ is the  orthonormal basis that diagonalizes the operator $A$,  $\{\beta_k(t)\}_{k \in\,\mathbb{Z}^2_0}$ is a sequence of independent Brownian motions defined on the stochastic basic $(\Omega,\mathcal{F},\{\mathcal{F}_t\}_{t\geq 0},\mathbb{P})$, and for any $\d>0$
\[\lambda_k(\d)=\le(1+\d\,|k|^{2\gamma}\r)^{-\frac 12},\ \ \ \ k \in\,\mathbb{Z}^2_0,\]
for some fixed $\gamma>0$. In other words, $w^\d$ is a Wiener process on $H$ with covariance $Q_\d=(I+\d (-A)^{\gamma})^{-1}$.
 We would like to stress that our result easily generalize to more general covariance operators.
 
As we mentioned above,  in the present paper we are interested in the asymptotic behavior of equation \eqref{eq1}, as both $\e$ and $\d$ go to zero. In particular, we shall assume that $\d$ is a function of $\e$, such that 
\[\lim_{\e\to 0}\d(\e)=0.\] In what follows we shall denote by $Q_\e$ the bounded linear operator in $H$ defined by
\[Q_\e e_k=\la_k(\d(\e))\,e_k,\ \ \ \ k \in\,\nat.\]

Now, if we project equation \eqref{eq1} on $H$, with the notations we have just introduced, it can be rewritten as
\begin{equation}
\label{eq1-abstract}
du(t)=\le[A u(t)+b(u(t))\r]\,dt+\sqrt{\e}\,dw^{\d(\e)}(t),\ \ \ \ t\geq 0,\ \ \ \ \ u(0)=u_0.
\end{equation}
As proven e.g.in  \cite{flandoli}, equation \eqref{eq1-abstract} admits a unique {\em generalized} solution $u_\e \in\,C([0,T];H)$. This means that $u_\e$ is a progressively measurable process taking values in $C([0,T];H)$,  such that $\mathbb{P}$-a.s. equation \eqref{eq1-abstract} is satisfied in the integral form
\[\le<u_\e(t),\varphi\r>_H=\le<u_0,\varphi\r>_H+\int_0^t\le<u_\e(s),A\varphi\r>_H\,ds+\int_0^t\le<b(u_\e(s),\varphi),u_\e(s)\r>_H\,ds+\sqrt{\e}\le<w^{\d(\e)}(t),\varphi\r>_H,\]
for every $t \in\,[0,T]$ and $\varphi \in\,D(A)$.

In what follows, for every $\a\geq 0$ and $\e>0$, we consider the auxiliary Ornstein-Uhlenbeck  problem
\begin{equation}
\label{linear}
dz(t)=(A-\a)z(t)\,dt+\sqrt{\e}\,dw^{\,\d(\e)}(t),\ \ \ \ t\geq 0,
\end{equation}
whose unique stationary solution is given by
\begin{equation}
\label{stoconv}
z_{\e}^\a(t)=\sqrt{\e}\int_{-\infty}^t e^{(t-s)(A-\a)}\,d\bar{w}^{\,\d(\e)}(s),\ \ \ \ t \in\,\reals.
\end{equation}
Notice that here $\bar{w}^{\,\d(\e)}(t)$ is a two sided cylindrical Wiener process, defined by
\[\bar{w}^{\,\d(\e)}(t,x)=\sum_{k \in\,\mathbb{Z}^2_0}\la_k(\d(\e))e_k(x)\,\bar{\beta}_k(t),\ \ \ (t,x) \in\,\reals\times D,\]
where
\[\bar{\beta}_k(t)=\le\{
\begin{array}{ll}
\ds{\beta_k(t),}  &  \ds{\text{if}\ t\geq 0,}\\
&  \vs
\ds{\tilde{\beta}_k(-t),} &   \ds{\text{if}\ t<0,}
\end{array}\r.\]
for some sequence of independent Brownian motions $\{\tilde{\beta}_k(t)\}_{k \in\,\mathbb{Z}^2_0}$,  defined on the stochastic basis  $(\Omega,\mathcal{F},\{\mathcal{F}_t\}_{t\geq 0},\mathbb{P})$ and independent of the sequence $\{\beta_k(t)\}_{k \in\,\mathbb{Z}^2_0}$. 

It is well known that for any  fixed $\e>0$ the process $z^\a_\e$ belongs to $L^p(\Omega;C([0,T];D((-A)^{\beta}))$, for any $T>0$, $p\geq 1$ and $\beta<\gamma/2$. In the case $\a=0$, we shall set 
\begin{equation}
\label{ca90}
z_\e(t):=z^0_\e(t).
\end{equation}

\section{The problem and the method}
\label{sec3}

We are here interested in the study of the validity of a large deviation principle, as $\e\downarrow 0$, for   the family $\{\mathcal{L}(u_\e)\}_{\e \in\,(0,1)}$,  where $u_\e$ is the solution of the equation
\begin{equation}
\label{equation-abstract}
du(t)=\le[A u(t)+b(u(t))\r]\,dt+\sqrt{\e}\,dw^{\d(\e)}(t),\ \ \ \ t\geq 0,\ \ \ \ \ u(0)=u_0.
\end{equation}
Here and in what follows $T>0$ is fixed and $\e>0\mapsto \d(\e) >0$ is a function such that
\begin{equation}
\label{ca89}
\lim_{\e\to 0}\d(\e)=0.\end{equation}
We will prove that depending on the scaling we assume between $\e$ and $\d(\e)$, the family $\{\mathcal{L}(u_\e)\}_{\e \in\,(0,1)}$ satisfies a large deviation principle in $\mathcal{E}$,  where $\mathcal{E}$ is  a suitable space of trajectories on $[0,T]$,  taking values in some space of functions defined on the domain $D$ and containing $H$.

\begin{Theorem}
\label{teo1}
Let $\e\mapsto \d(\e)$ be a function satisfying \eqref{ca89}. Moreover, assume that there exists $\eta>0$ such that
\begin{equation}
\label{scaling1}
\lim_{\e\to 0} \e\,\d(\e)^{-\eta}=0.
\end{equation}
Then, for any $u_0 \in\,H$,  the family $\{\mathcal{L}(u_\e)\}_{\e>0}$ satisfies a large deviation principle in $C([0,T];H)$, with action functional
\begin{equation}
\label{action}
I_T(f)=\frac 12\int_0^T|f^\prime(t)-A f(t)-b(f(t))|_H^2\,dt.\end{equation}

\end{Theorem}

\begin{Theorem}
\label{teo2}
Let  $\e\mapsto \d(\e)$ be a function satisfying \eqref{ca89}. Moreover, let $\si<0$ and $p\geq 2$ be such that
\[
\si>-\frac 2p\vee \le(\frac 2p-1\r).
\]
Then, for any $u_0 \in\,H^\theta(D)$, with $\theta \geq \si+1-2/p$,  the family $\{\mathcal{L}(u_\e)\}_{\e>0}$ satisfies a large deviation principle in $C([0,T];\mathcal{B}^\si_{p}(D))$,  with the same action function $I_T$ introduced in \eqref{action}
 
  \end{Theorem}

In order to prove Theorems \ref{teo1} and \ref{teo2}, we follow the weak convergence approach, as developed in \cite{BDM}. To this purpose, we need to introduce some notations.
We  denote by $\mathcal{P}_T$ the set of predictable processes in $L^2(\Omega\times[0,T];H)$, and for any $T>0$ and $\gamma>0$, we define the sets
\[\mathcal{S}^\gamma_T=\le\{f \in\,L^2(0,T;H)\ :\ \int_0^T |f(t)|_H^2\,dt\leq \gamma\r\},\]
and
\[\mathcal{A}^\gamma_T=\le\{ u \in\,\mathcal{P}_T\ :\ u \in\,\mathcal{S}^\gamma_T,\ \ \mathbb{P}-\text{a.s.} \r\}.
\]

Next, for any predictable process $\varphi(t)$ taking values in $L^2([0,T];H)$, we denote by $u^\varphi_\e(t)$ the generalized solution of the problem
\begin{equation}
\label{controlled}
du(t)=\le[A u(t)+b(u(t))+Q_\e\,\varphi(t)\r]\,dt+\sqrt{\e}\,dw^{\d(\e)}(t),\ \ \ \ t\geq 0,\ \ \ \ \ u(0)=u_0.
\end{equation}
Moreover, we denote by $u^\varphi$ the solution of the problem 
\begin{equation}
\label{det}
\frac{du}{dt}(t)=A u(t)+b(u(t))+\varphi(t)\,\ \ \ \ t\geq 0,\ \ \ \ \ u(0)=u_0.
\end{equation}

As for equation \eqref{eq1-abstract}, for any fixed $\e\geq 0$ and for any $T>0$ and $\kappa \geq 1$,  equation \eqref{controlled} admits a unique generalized solution $u^\varphi_\e$ in $L^\kappa(\Omega;C([0,T];H))$. As a particular case ($\e=0$), we have also well-posedness for equation \eqref{det}.

\medskip

By proceeding as in  \cite{BDM}, the following result can be proven.
\begin{Theorem}
\label{teo-bd}

Let  $\mathcal{E}$ be  a Polish space of trajectories defined on $[0,T]$ with values in a space of functions defined on the domain $D$ and containing the space $H$, and let $I_T$ be the functional defined in \eqref{action}.
Assume that
\begin{enumerate}
\item the level sets $\{ I_T(f)\leq r\}$ are compact in $\mathcal{E}$, for every $r\geq 0$;
\item 
for every  family $\{\varphi_\e\}_{\e>0}\subset {\mathcal A}^\gamma_T$ that converges in distribution, as $\e\downarrow 0$, to some $\varphi \in\,{\mathcal A}^\gamma_T$,  in the space $L^2(0,T;H)$, endowed with the weak topology, the family $\{u^{\varphi_\e}_\e\}_{\e>0}$ converges in distribution to $u^\varphi$, as $\e\downarrow 0$, in $\mathcal{E}$.
\end{enumerate}
Then the family $\{{\mathcal L}(u_\e)\}_{\e>0}$ satisfies a large deviation principle in $\mathcal{E}$, with action functional $I_T$.
\end{Theorem}

Actually, as shown in \cite{BDM}, the convergence of $u_\e^{\varphi_\e}$ to $u^\varphi$ implies the validity of the Laplace principle in $\mathcal{E}$ with rate functional $I_T$. This means that, for any continuous mapping $\Gamma:\mathcal{E}\to \reals$ it holds
\begin{equation}
\label{laplace}
\lim_{\e\to 0}-\e\log\E\,\exp\le(-\frac 1\e\,\Gamma(u_\e)\r)=\inf_{f \in\,\mathcal{E}}\le(\,\Gamma(f)+I_T(f)\,\r).\end{equation}
And, once one has shown that the level sets of $I_T$ are compact in $\mathcal{E}$, the validity of the Laplace principle as  in \eqref{laplace} is equivalent to say that the family $\{{\mathcal L}(u_\e)\}_{\e>0}$ satisfies a large deviation principle in $\mathcal{E}$, with action functional $I_T$.

The proof of condition 1 in Theorem \ref{teo-bd} is obtained once we show that, when  the space $L^2(0,T;H)$ is endowed with the topology of weak convergence, the mapping
\[\varphi \in\,L^2(0,T;H)\mapsto u^\varphi \in\,\mathcal{E},\]
is continuous. More precisely, condition 1 will follow if we can prove that for any sequence   $\{\varphi_n\}_{n \in\,\mathbb{N}}$   in $L^2(0,T;H)$, weakly convergent to some $\varphi \in\,L^2(0,T;H)$, it holds
 \[\lim_{n\to \infty}|u^{\varphi_n}-u^\varphi|_{\mathcal{E}}=0.\]

 As for condition 2, we will use Skorohod theorem and rephrase such a condition in the following way. Let $(\bar{\Omega}, \bar{\mathcal{F}},\bar{\mathbb{P}})$ be a probability space and let $\{\bar{w}^{\d(\e)}(t)\}_{t\geq 0}$ be a  Wiener process, with covariance $Q_\d$, defined on such a probability space and corresponding to the filtration $\{\bar{\mathcal{F}}_t\}_{t\geq 0}$. Moreover, let $\{\bar{\varphi}_\e\}_{\e>0}$ and $\bar{\varphi}$ be  $\{\bar{\mathcal{F}}_t\}_{t\geq 0}$-predictable processes taking values in $\mathcal{S}^\gamma_T$, $\bar{\mathbb{P}}$ almost surely, such that  the distribution of $(\bar{\varphi}_{\e}, \bar{\varphi}, \bar{w}^{\,\d(\e)})$ coincides with the distribution of $(\varphi_\e,\varphi,w^{\,\d(\e)})$ and
 \[\lim_{\e\to 0}\bar{\varphi}_\e=\bar{\varphi}\ \ \ \text{weakly in } L^2(0,T;H),\ \ \ \ \bar{\mathbb{P}}-\text{a.s.}\]
 Then, if $\bar{u}^{\,\bar{\varphi}_\e}_\e$ is the solution of an equation analogous to \eqref{controlled}, with $\varphi_\e$ and $w^{\,\d(\e)}$ replaced respectively by $\bar{\varphi}_\e$ and $\bar{w}^{\,\d(\e)}$, we have that 
 \begin{equation}
 \label{ca155}
 \lim_{\e\to 0} \bar{u}^{\,\bar{\varphi}_\e}_\e=\bar{u}^{\,\bar{\varphi}}\ \ \ \text{in }\mathcal{E},\ \ \ 
\overline{\mathbb{P}}-\text{a.s.}\end{equation}

We would like to stress that condition 1 in Theorem \ref{teo-bd} follows from condition 2. Actually, if we take in equation \eqref{controlled} $\sqrt{\e}=0$ and  $\{\varphi_\e\}_{\e>0}=\{\varphi_n\}_{n \in\,\nat}$ and $\varphi$  deterministic, then condition 1 is a particular case of condition 2.

 \section{Proof of Theorem \ref{teo1}}
\label{sec4}

In what follows, $\{\varphi_\e\}_{\e \in\,(0,1)}$ and $\varphi$ are predictable processes in $\mathcal{A}^\gamma_T$, for some $\gamma>0$ fixed, such that $\varphi_\e$ converges to $\varphi$, $\mathbb{P}$ almost surely, in the weak topology of $L^2(0,T;H)$.

For any $\a\geq 0$ and $\e>0$, we introduce the random equation
\begin{equation}
\label{random}
\frac{dv}{dt}(t)=A v(t)+b(v(t)+z^\a_\e(t)+\Phi_\e(t))+\a\,z^\a_\e(t),\ \ \ \ v(0)=u_0-z^\a_\e(0),
\end{equation}
where
$z^\a_\e$  is the process introduced in \eqref{stoconv}, solution of the linear equation \eqref{linear}, and
\[\Phi_\e(t)=\int_0^te^{(t-s)A}Q_\e\,\varphi_\e(s)\,ds,\ \ \ t\geq 0,\]
is the solution of the problem 
\[\frac{d\Phi_\e}{dt}(t)=A\Phi_\e(t)+Q_\e\,\varphi_\e(t),\ \ \ \ \Phi_\e(0)=0.\]
Notice that if $\varphi_\e \in\,\mathcal{A}^\gamma_T$, for some $\gamma>0$, then
\[\le|\Phi_\e(t)\r|_{L^p(D)}\leq c\int_0^t(t-s)^{-\frac{p-2}{2p}}|\varphi_\e(s)|_H\,ds,\]
so that
\[\begin{array}{l}
\ds{\le|\Phi_\e\r|^p_{L^p(0,T;L^p(D))}\leq c\int_0^T\le(\int_0^t(t-s)^{-\frac{p-2}{2p}}|\varphi_\e(s)|_H\,ds\r)^p\,dt}\\
\vs
\ds{\leq c_T\,|\varphi_\e|_{L^2(0,T;H)}^p \le(\int_0^T s^{-\frac{p-2}{p+2}}\,ds\r)^{\frac{p+2}2}.}
\end{array}\]

This implies that
\begin{equation}
\label{ca15}
\ds{\le|\Phi_\e\r|_{L^p(0,T;L^p(D))}\leq c_{T,p}\,\sqrt{\gamma},\ \ \ \mathbb{P}-\text{a.s}.}
\end{equation}

As shown e.g. in \cite{flandoli}, equation \eqref{random} admits a unique solution 
\begin{equation}
\label{regularity}
v^\a_\e \in\,C([0,T];H)\cap L^2(0,T;V),
\end{equation}
and the unique generalized solution $u_\e^\a$ of equation 
\begin{equation}
\label{controlled-eps}
du(t)=\le[A u(t)+b(u(t))+Q_\e\,\varphi_\e(t)\r]\,dt+\sqrt{\e}\,dw^{\d(\e)}(t),\ \ \ \ t\geq 0,\ \ \ \ \ u(0)=u_0,
\end{equation}
can be decomposed as 
\[u^\a_\e(t)=v^\a_\e(t)+z^\a_\e(t)+\Phi_\e(t),\ \ \ \ t \in\,[0,T].\]

\begin{Lemma}
Assume that $\{\varphi_\e\}_{\e>0}\subset \mathcal{A}^\gamma_T$, for some fixed $\gamma>0$. Then, there exists $c_{T,\gamma}>0$ such that for every $\e>0$ and $t \in\,[0,T]$ 
\begin{equation}
\label{ca10}
\begin{array}{l}
\ds{
|v_\e^\a(t)|_H^2+\int_0^t|v^\a_\e(s)|_V^2\,ds}\\
\vs
\ds{\leq c_{T,\gamma}\exp\le(c\,|z_\e^\a|_{L^4(0,t;L^4(D))}^4\r)\le(|u_0|_H^2+|z^\a_\e(0)|_H^2+(\a^2+1)\,|z_\e^\a|_{L^4(0,t;L^4(D))}^4+1\r).}
\end{array}
\end{equation}

Moreover, we have
\begin{equation}
\label{ca32}
\begin{array}{l}
\ds{|v^\a_\e|^4_{L^4(0,T;L^4(D))}}\\
\vs
\ds{\leq c_{T,\gamma}\exp\le(c\,|z_\e^\a|_{L^4(0,t;L^4(D))}^4\r)\le(|u_0|_H^2+|z^\a_\e(0)|_H^2+(\a^2+1)\,|z_\e^\a|_{L^4(0,t;L^4(D))}^4+1\r)^2.}
\end{array}
\end{equation}

\end{Lemma}

\begin{proof}
Let $v^\a_\e$ be the solution of problem \eqref{random}, having the regularity specified in  \eqref{regularity}. Due to the first identity in \eqref{identities}, we have
\[\begin{array}{l}
\ds{\frac 12\frac d{dt}|v^\a_\e(t)|_H^2+|v^\a_\e(t)|_V^2}\\
\vs
\ds{=\le<b(z^\a_\e(t)+\Phi_\e(t)),v^\a_\e(t)\r>_H+\le<b(v^\a_\e(t),z^\a_\e(t)+\Phi_\e(t)),v^\a_\e(t)\r>_H+\a\le<z^\a_\e(t),v^\a_\e(t)\r>_H.}
\end{array}\]

For every $\eta>0$, we have
\[\begin{array}{l}
\ds{\le|\le<b(z^\a_\e(t)+\Phi_\e(t)),v^\a_\e(t)\r>_H\r|=\le|\le<b(z^\a_\e(t)+\Phi_\e(t),v^\a_\e(t)),z^\a_\e(t)+\Phi_\e(t)\r>_H\r|}\\
\vs
\ds{\leq |v^\a_\e(t)|_V\,|z^\a_\e(t)+\Phi_\e(t)|_{L^4(D)}^2\leq \eta\,|v^\a_\e(t)|_V^2+c_\eta \le(|z^\a_\e(t)|_{L^4(D)}^4+|\Phi_\e(t)|_{L^4(D)}^4\r).}
\end{array}\]
As $H^{1/2}(D)\hookrightarrow L^4(D)$, by interpolation, we have
\[\begin{array}{l}
\ds{\le|\le<b(v^\a_\e(t),z^\a_\e(t)+\Phi_\e(t)),v^\a_\e(t)\r>_H\r|=\le|\le<b(v^\a_\e(t)),z^\a_\e(t)+\Phi_\e(t) \r>_H\r|}\\
\vs
\ds{\leq c|v^\a_\e(t)|_V\,|v^\a_\e(t)|_{H^{1/2}}|z^\a_\e(t)+\Phi_\e(t)|_{L^4(D)}\leq c|v^\a_\e(t)|^{3/2}_V\,|v^\a_\e(t)|_{H}^{1/2}|z^\a_\e(t)+\Phi_\e(t)|_{L^4(D)}}\\
\vs
\ds{\leq \eta\,|v^\a_\e(t)|_V^2+c_\eta\,|v^\a_\e(t)|_H^2\le(|z^\a_\e(t)|_{L^4(D)}^4+|\Phi_\e(t)|_{L^4(D)}^4\r).}
\end{array}\]
Moreover, we have 
\[\a|\le<z^\a_\e(t),v^\a_\e(t)\r>_H|\leq \eta\,|v^\a_\e(t)|^2_V+c_\eta\, \a^2\,|z^\a_\e(t)|_{H^{-1}}^2.\]
Therefore, if we pick $\eta=1/6$,  we get
\[\begin{array}{l}
\ds{\frac d{dt}|v^\a_\e(t)|_H^2+|v^\a_\e(t)|_V^2}\\
\vs
\ds{\leq c\,|v^\a_\e(t)|_H^2\le(|z^\a_\e(t)|_{L^4(D)}^4+|\Phi_\e(t)|_{L^4(D)}^4\r)+c\,(\a^2+1)\,|z^\a_\e(t)|_{L^4(D)}^4+c\,|\Phi_\e(t)|_{L^4(D)}^4.}
\end{array}\]
Due to \eqref{ca15}, by using the Gronwall lemma this yields \eqref{ca10}.

In order to prove \eqref{ca32}, we notice that, as $H^{1/2}(D)\hookrightarrow L^4(D)$, by interpolation we have
\[\begin{array}{l}
\ds{|v^\a_\e|^4_{L^4(0,T;L^4(D))}\leq c\int_0^T |v^\a_\e(s)|_V^2\,|v^\a_\e(s)|_H^2\,ds \leq |v^\a_\e|^2_{L^\infty(0,T;H)}|v^\a_\e|^2_{L^2(0,T;V)}.}
\end{array}\]
Therefore, \eqref{ca32} follows immediately from \eqref{ca10}.

\end{proof}

\begin{Remark}
{\em  \begin{enumerate}

\item  Due to \eqref{ca40}, there exist $\bar{\kappa}\geq 1$ and $c(T)>0$ such that for any $\e>0$
\begin{equation}
\label{ca44}
\a_\e:=c(T)\,|K_\e(4,\beta_\eta)|^{\bar{\kappa}}\vee 1\Longrightarrow |z^{\a_\e}_{\e}|_{L^4(0,T;L^4(D))}\leq 1 \mbox{ and } |z^{\a_\e}_{\e}(0)|_{H}\leq 1.
\end{equation}
Thanks to  \eqref{ca32}, this implies that
\begin{equation}
\label{ca60}
|v^{\a_\e}_\e|_{L^4(0,T;L^4(D))}\leq c_{T,\gamma}\le(|u_0|_H+\a_\e^2+1\r),\ \ \ \ \mathbb{P}-\text{a.s.}\end{equation}
and in view of \eqref{sa2}, we can conclude that if \eqref{scaling1} holds, then 
\begin{equation}
\label{ca43}
\E\,|v^{\a_\e}_\e|_{L^4(0,T;L^4(D))}^\kappa\leq c_{\gamma}(T,\kappa)\le(|u_0|_H^\kappa+ 1\r),\ \ \ \ \kappa\geq 1.
\end{equation}

\item As a consequence of \eqref{ca10}, if $\varphi \in\,\mathcal{A}_T^\gamma$ and $v^\varphi$ is a solution to the problem
\[
\frac{dv}{dt}(t)=A v(t)+b(v(t)+\Gamma(\varphi)(t)),\ \ \ \ v(0)=u_0,
\]
where 
\[\Gamma(\varphi)(t):=\int_0^te^{(t-s)A}\varphi(s)\,ds,\]
we have
\begin{equation}
\label{det-est}
|v^\varphi(t)|_H^2+\int_0^t|v^\varphi(s)|_H^2\,ds\leq c_{T,\gamma}\le(1+|u_0|_H^2\r).
\end{equation}

Moreover,  by interpolation, 
\begin{equation}
\label{det-est-l4}
|v^\varphi|_{L^4(0,T;L^4(D))}\leq c_{T,\gamma}(1+|u_0|_H).
\end{equation}

\end{enumerate}}
\end{Remark}

In the next lemma we investigate the continuity properties of the operator $\Gamma$ and we prove the convergence of $\Phi_\e$ to 
$\Gamma(\varphi)$ in case the sequence $\{\varphi_\e\}_{\e>0}$ is weakly convergent to $\varphi$.

\begin{Lemma}
For every $\rho <1$ there exists $\theta_\rho>0$ such that
\begin{equation}
\label{ca20}
|\Gamma(\varphi)|_{C^{\theta_\rho}([0,T];H^\rho(D))}\leq c_\rho\,|\varphi|_{L^2(0,T;H)},
\end{equation}
for every $\varphi \in\,L^2(0,T;H)$.
In particular, if $\{\varphi_\e\}_{\e>0}$  is a family in $ \mathcal{A}^\gamma_T$, weakly convergent in $L^2(0,T;H)$ to some $\varphi \in\, \mathcal{A}^\gamma_T$,
for every $\rho< 1$ we have
\begin{equation}
\label{ca30}
\lim_{\e\to 0}|\Phi_\e-\Gamma(\varphi)|_{C([0,T];H^\rho(D))}=0.
\end{equation}
\end{Lemma}

\begin{proof}
For every $\beta \in\,(0,1)$, we have
\[\begin{array}{l}
\ds{\Gamma(\varphi)(t)=c_\beta\int_0^t (t-s)^{-\beta+1}e^{(t-s)A}Y_\beta(\varphi)(s)\,ds,}
\end{array}\]
where
\[Y_\beta(\varphi)(s)=\int_0^s(s-\si)^{-\beta} e^{(s-\si)A} \varphi(\si)\,d\si.\]
Due to the Young inequality, we get
\[\begin{array}{l}
\ds{|Y_\beta(\varphi)|_{L^p(0,T;H)}^p=\int_0^T\le(\int_0^s (s-\si)^{-\beta}|\varphi(\si)|_H\,d\si\r)^p\,ds\leq |\varphi|_{L^2(0,T;H)}^p\le(\int_0^T s^{-\frac{2\beta p}{p+2}}\,ds\r)^{\frac{p+2}p},}
\end{array}\]
and hence, if $\beta<1/2+1/p$, we have
\[|Y_\beta(\varphi)|_{L^p(0,T;H)}\leq c_p(T)\,|\varphi|_{L^2(0,T;H)}.\]

Now, as shown e.g. in \cite{dpz-erg}, if $\beta>\rho/2+1/p$ we have that the mapping
\[Y \in\,L^p(0,T;H)\mapsto \int_0^t (t-s)^{-\beta+1}e^{(t-s)A} Y(s)\,ds \in\,C^{\beta-\frac \rho 2-\frac 1p}([0,T];H^\rho(D)),\]
is continuous. Therefore, we can conclude that
\[|\Gamma(\varphi)|_{C^{\beta-\frac \rho 2-\frac 1p}([0,T];H^\rho(D))}\leq c_{\rho,\beta}(T)\,|\varphi|_{L^2(0,T;H)},\]
if
$\rho/2+1/p<\beta<1/2+1/p$, and this implies \eqref{ca20}.

Now, in order to prove \eqref{ca30}, we notice that 
\[\Phi_\e-\Gamma(\varphi)=\Gamma(Q_\e(\varphi_\e-\varphi))+\Gamma(Q_\e\varphi-\varphi).\]
Since $Q_\e(\varphi_\e-\varphi) \in\,\mathcal{A}^\gamma_T$ and $Q_\e(\varphi_\e-\varphi)\rightharpoonup 0$, as $\e\downarrow 0$,  weakly  in $L^2(0,T;H)$,  due to the compactness of the immersion of $C^{\theta_{\rho_1}}([0,T];H^{\rho_1}(D))$ into $C([0,T];H^{\rho_2}(D))$, for every $\rho_1>\rho_2$, from  \eqref{ca20} we conclude  that 
\begin{equation}
\label{ca31}
\lim_{\e\to 0}|\Gamma(Q_\e(\varphi_\e-\varphi))|_{C([0,T];H^\rho(D))}=0,\end{equation}
for every $\rho<1$.
Moreover, thanks again to \eqref{ca20}, 
 \[\begin{array}{l}
\ds{|\Gamma(Q_\e\,\varphi-\varphi)|_{C([0,T];H^\rho(D))}^p\leq c_\rho|Q_\e\,\varphi-\varphi|_{L^2(0,T;H)}\to 0,\ \ \ \ \text{as}\ \e\to 0,}
\end{array}\]
and together with \eqref{ca31}, this implies \eqref{ca30}.
\end{proof}

In what follows, we shall denote
\[\rho_\e^\a(t):=v^\a_\e(t)-v^\varphi(t),\ \ \ \ t\geq 0.\]
 It is immediate to check that $\rho_\e^\a$ is a solution to the problem
\begin{equation}
\frac{d\rho^\a_\e}{dt}(t)=A \rho^\a_\e(t)+b(v^\a_\e(t)+z^\a_\e(t)+\Phi_\e(t))-b(v^\varphi(t)+\Gamma(\varphi)(t))+\a\,z^\a_\e(t),\ \ \ \ \rho^\a_\e(0)=-z^\a_\e(0).
\end{equation}

\begin{Lemma}
If $\{\varphi_\e\}_{\e>0}\subset \mathcal{A}^\gamma_T$ and $\varphi \in\,\mathcal{A}_T^\gamma$, for every $\a\geq 0$ we have
\begin{equation}
\label{ca35}
\begin{array}{l}
\ds{\sup_{t \in\,[0,T]}|\rho^\a_\e(t)|_H^2+\int_0^T|\rho^\a_\e(t)|_V^2\,dt\leq c_\gamma(T)\,\exp\le(u_0|^4_{H}+1\r)}\\
 \vs
 \ds{\le(|z^\a_\e(0)|_H^2+|z^\a_\e|_{L^4(0,T;L^4(D))}^2\le(\,|v^\a_\e|_{L^4(0,T;L^4(D))}^2+1+\a^2\r)+|z^\a_\e|_{L^4(0,T;L^4(D))}^4\r.}\\
 \vs
 \ds{\le.+|\Phi_\e-\Gamma(\varphi)|_{L^4(0,T;L^4(D))}^2\le(1+|u_0|_{H}^2+|v^\a_\e|_{L^4(0,T;L^4(D))}^2\r)\r).}
 \end{array}
\end{equation}
\end{Lemma}

\begin{proof}
Taking into account of the first identity in \eqref{identities}, we have
\[\begin{array}{l}
\ds{\frac 12 \frac d{dt}|\rho^\a_\e(t)|_H^2+|\rho^\a_\e(t)|_V^2=\le<b(v^\a_\e(t))-b(v^\varphi(t)),\rho^\a_\e(t)\r>_H+\le<b(\Phi_\e(t))-b(\Gamma(\varphi)(t)),\rho^\a_\e(t)\r>_H}\\
\vs
\ds{+\le<b(z^\a_\e(t)),\rho^\a_\e(t)\r>_H+\le<b(v^\a_\e(t),z^\a_\e(t))+b(z^\a_\e(t),v^\a_\e(t)),\rho^\a_\e(t)\r>_H}\\
\vs
\ds{+\le<b(z^\a_\e(t),\Phi_\e(t))+b(\Phi_\e(t),z^\a_\e(t)),\rho^\a_\e(t)\r>_H+\le<b(\rho^\a_\e(t),\Phi_\e(t)),\rho^\a_\e(t)\r>_H}\\
\vs
\ds{+\le<b(v^\varphi(t),\Phi_\e(t)-\Gamma(\varphi)(t))+b(\Phi_\e(t)-\Gamma(\varphi)(t),v^\a_\e(t)),\rho^\a_\e(t)\r>_H+\a\,\le<z^\a_\e(t),\rho^\a_\e(t)\r>_H}\\
\vs
\ds{:=\sum_{j=1}^8 I^\a_{\e,j}(t).}
\end{array}\]
Now, we are going to estimate each one of the terms $I^\a_{\e,j}(t)$, for $j=1,\ldots,8$. We have
\[I^\a_{\e,1}(t)=\le<b(\rho^\a_\e(t),v^\varphi(t),\rho^\a_\e(t)\r>_H=-\le<b(\rho^\a_\e(t)),v^\varphi(t)\r>_H,\]
so that, by interpolation, for any $\eta>0$,
\begin{equation}
\label{ca-I_1}
|I^\a_{\e,1}(t)|\leq |\rho^\a_\e(t)|_V|\rho^\a_\e(t)|_{L^4(D)}|v^\varphi(t)|_{L^4(D)}\leq \eta\,|\rho^\a_\e(t)|_V^2+c_\eta\,|\rho^\a_\e(t)|_H^2\,|v^\varphi(t)|^4_{L^4(D)}.\end{equation}
For $I^\a_{\e,2}(t)$ we have
\[\begin{array}{l}
\ds{\le<b(\Phi_\e(t))-b(\Gamma(\varphi)(t)),\rho^\a_\e(t)\r>_H}\\
\vs
\ds{=\le<b(\Phi_\e(t),\Phi_\e(t)-\Gamma(\varphi)(t))+b(\Phi_\e(t)-\Gamma(\varphi)(t),\Gamma(\varphi)(t)),\rho^\a_\e(t)\r>_H,}
\end{array}\]
and, by proceeding as for $I^\a_{\e,1}(t)$, we have
\begin{equation}
\label{ca-I_2}
|I^\a_{\e,2}(t)|\leq \eta\,|\rho^\a_\e(t)|_V^2+c_{\eta}\le(|\Phi_\e(t)|_{L^4(D)}^2+|\Gamma(\varphi)(t)|_{L^4(D)}^2\r)|\Phi_\e(t)-\Gamma(\varphi)(t)|_{L^4(D)}^2.
\end{equation}
For $I^\a_{\e,3}(t)$, we have
\begin{equation}
\label{ca-I_3}
|I^\a_{\e,3}(t)|=\le|\le<b(z^\a_\e(t)),\rho^\a_\e(t)\r>_H\r|\leq \eta\,|\rho^\a_\e(t)|_V^2+c_\eta\,|z^\a_\e(t)|_{L^4(D)}^4,
\end{equation}
and, in an analogous way, 
\begin{equation}
\label{ca-I_4}
|I^\a_{\e,4}(t)|+|I^\a_{\e,5}(t)|\leq \eta\,|\rho^\a_\e(t)|_V^2+c_\eta\,|z^\a_\e(t)|_{L^4(D)}^2\le(\,|v^\a_\e(t)|_{L^4(D)}^2+|\Phi_\e(t)|_{L^4(D)}^2\r).
\end{equation}
Concerning $I^\a_{\e,6}(t)$, by interpolation we get
\begin{equation}
\label{ca-I_6}
|I^\a_{\e,6}(t)|\leq \eta\,|\rho^\a_\e(t)|_V^2+c_\eta\,|\Phi_\e(t)|_{L^4(D)}^4\,|\rho^\a_\e(t)|_H^2.
\end{equation}
Finally, with the same arguments used for $I^\a_{\e,3}$, and also for $I^\a_{\e,4}$ and $I^\a_{\e,5}$, we get
\begin{equation}
\label{ca-I_7}
I^\a_{\e,7}(t)|\leq \eta\,|\rho^\a_\e(t)|_V^2+c_\eta\le(|v^\varphi|_{L^4(D)}^2+|v^\a_\e(t)|_{L^4(D)}^2\r)|\Phi_\e(t)-\Gamma(\varphi)|_{L^4(D)}^2.
\end{equation}
For the last term, we have
\begin{equation}
|I^\a_{\e,8}(t)|\leq \eta\,|\rho^\a_\e(t)|_V^2+c_\eta\,\a^2\,|z^\a_\e(t)|^2_{H^{-1}}.
\end{equation}
Therefore, if we take $\eta=1/14$, we obtain
\[ \begin{array}{l}
\ds{\frac d{dt}|\rho^\a_\e(t)|_H^2+|\rho^\a_\e(t)|_V^2\leq c\,|\rho^\a_\e(t)|_H^2\le( |v^\varphi(t)|^4_{L^4(D)}+|\Phi_\e(t)|_{L^4(D)}^4\r)}\\
\vs
\ds{+\le(|\Phi_\e(t)|_{L^4(D)}^2+|\Gamma(\varphi)(t)|_{L^4(D)}^2+|v^\varphi(t)|_{L^4(D)}^2+|v^\a_\e(t)|_{L^4(D)}^2\r)|\Phi_\e(t)-\Gamma(\varphi)(t)|_{L^4(D)}^2}\\
\vs
\ds{+c\,|z^\a_\e(t)|_{L^4(D)}^2\le(\,|v^\a_\e(t)|_{L^4(D)}^2+|\Phi_\e(t)|_{L^4(D)}^2+\a^2\r)+c\,|z^\a_\e(t)|_{L^4(D)}^4.}
\end{array}\]
 Recalling that
 \[\varphi \in\,\mathcal{A}^\gamma_T\Longrightarrow |\Gamma(\varphi)|_{L^p(0,T;L^p(D))}\leq c_p(T)\, \gamma,\ \ \ \ \mathbb{P}-\text{a.s.}\]
as a consequence of the Gronwall lemma,  this implies that
 \[\begin{array}{l}
 \ds{\sup_{t \in\,[0,T]}|\rho^\a_\e(t)|_H^2+\int_0^T|\rho^\a_\e(t)|_V^2\,dt\leq c_\gamma(T)\,\exp\le(|v^\varphi|^4_{L^4(0,T;L^4(D))}\r)}\\
 \vs
 \ds{\le(|z^\a_\e(0)|_H^2+|z^\a_\e|_{L^4(0,T;L^4(D))}^2\le(\,|v^\a_\e|_{L^4(0,T;L^4(D))}^2+1+\a^2\r)+|z^\a_\e|_{L^4(0,T;L^4(D))}^4\r.}\\
 \vs
 \ds{\le.+|\Phi_\e-\Gamma(\varphi)|_{L^4(0,T;L^4(D))}^2\le(1+|v^\varphi|_{L^4(0,T;L^4(D))}^2+|v^\a_\e|_{L^4(0,T;L^4(D))}^2\r)\r).}
 \end{array}\]
Thanks to \eqref{det-est-l4}, we conclude that \eqref{ca35} holds.

\end{proof}

\subsection{Conclusion of the proof of Theorem \ref{teo1}}

We have already seen that, if $\a$ is any given non-negative constant and $v^\a_\e(t)$ is the solution to problem \eqref{random},  then it holds
\[u_\e^{\varphi_\e}(t)=v^\a_\e(t)+z^\a_\e(t)+\Phi_\e(t),\ \ \ \ t\geq 0.\]
Since $u^\varphi(t)=v^\varphi(t)+\Gamma(\varphi)(t)$, this implies that we can write
\[u_\e^{\varphi_\e}(t)-u^\varphi(t)=\le[v^{\a_\e}_\e(t)-v^\varphi(t)\r]+z^{\a_\e}_\e(t)+\le[\Phi_\e(t)-\Gamma(\varphi)(t)\r],\ \ \ \ t\geq 0,\]
where   $\a_\e$ is the random constant  defined in \eqref{ca44}.

Due to  \eqref{ca60} and \eqref{ca35}, it is immediate to check that
\[\begin{array}{l}
\ds{|v^{\a_\e}_\e(t)-v^\varphi(t)|_H^2\leq c_\gamma(T,|u_0|_H)\le[|z^{\a_\e}_\e(0)|_H^2+|z^{\a_\e}_\e|_{L^4(0,T;L^4(D))}^4\r.}\\
\vs
\ds{\le.+\le(|z^{\a_\e}_\e|_{L^4(0,T;L^4(D))}^2+|\Phi_\e-\Gamma(\varphi)|_{L^4(0,T;L^4(D))}\r)\le(1+\a_\e^2\r)\r].}
\end{array}\]
Now, in view of \eqref{ca40}, for any $\beta \in\,(0,1/4)$ there exists $c_\beta(T)$ such that for every $\a>0$
\[|z^\a_\e|_{C([0,T];L^4(D))}\leq c_\beta(T)\,K_\e(4,\beta),\ \ \ \ \mathbb{P}-\text{a.s.}\]
This implies that, if we fix any $\eta \in\,(0,1/2\gamma)$ satisfying \eqref{scaling1} and $\beta_\eta \in\,(0,1/4)$ so that \eqref{sa2} holds, we get
\begin{equation}
\label{mc1}
\begin{array}{l}
\ds{|v^{\a_\e}_\e(t)-v^\varphi(t)|_H^2\leq c_{\gamma,\eta}(T,|u_0|_H)\le(K_\e^4(4,\beta_\eta)+K^2_\e(2,\beta)+|\Phi_\e-\Gamma(\varphi)|_{L^4(0,T;L^4(D))}\r)\le(1+\a_\e^2\r).}
\end{array}\end{equation}
As a consequence of \eqref{sa2} and assumption \eqref{scaling1}, we have 
\[\sup_{\e \in\,(0,1)}\E\,\a_\e^\kappa<\infty,\ \ \ \ \kappa\geq 1.\]
Then, thanks again to \eqref{sa2}, from \eqref{mc1} we can conclude that for any $\kappa\geq 1$
\[\E\,|v^{\a_\e}_\e-v^\varphi|_{C([0,T];H)}^\kappa\leq c_{\gamma,\eta,\kappa}(T, |u_0|_H)\le[\le(\e\,\d(\e)^{-\eta}\r)^{c_\kappa}+\le(\E\,|\Phi_\e-\Gamma(\varphi)|^\kappa_{L^4(0,T;L^4(D))}\r)^{\frac 12}\r].\]
Because of \eqref{scaling1}, \eqref{ca20} and \eqref{ca30}, this implies that 
\begin{equation}
\label{ca80}
\lim_{\e \to 0}
\e\,\d(\e)^{-\eta}=0\Longrightarrow \lim_{\e\to 0}\E\,|v^{\a_\e}_\e-v^\varphi|_{C([0,T];H)}^\kappa=0,\ \ \ \ \kappa\geq 1.
\end{equation}

Since 
\[|u_\e^{\varphi_\e}-u^\varphi|_{C([0,T];H)}\leq |v^{\a_\e}_\e-v^\varphi|_{C([0,T];H)}+|z^{\a_\e}_\e|_{C([0,T];H)}+|\Phi_\e-\Gamma(\varphi)|_{C([0,T];H)},\]
\eqref{ca80},
together once more with   \eqref{ca20} and \eqref{ca30},  implies that 
\begin{equation}
\label{ca82}
\lim_{\e\to 0}
\e\,\d(\e)^{-\eta}=0\Longrightarrow\lim_{\e\to 0}|u_\e^{\varphi_\e}-u^\varphi|_{C([0,T];H)}^\kappa=0,\ \ \ \ \ \kappa\geq 1.
\end{equation}

In view of Theorem \ref{teo-bd} and all comments in Section \ref{sec3} after Theorem \ref{teo-bd}, we can conclude that Theorem \ref{teo1} is proved.

\section{Proof of Theorem \ref{teo2}}
\label{sec5}

In what follows, we  fix any $\si<0$ and $p\geq 2$ such that
\[\si>-\frac 2p\vee \le(\frac 2p-1\r).
\] Because of such a condition, we can fix two real constants $\a$ and $\beta$ such that 
\begin{equation}
\label{ca103}
\frac 2p>\a>-\si>0,\ \ \ p\geq 2,\ \ \ \beta\geq 2,\ \ \ -\frac 12+\frac 1p<\frac \a 2-\frac 1\beta<\frac \si 2.\end{equation}
Once fixed $\a$, $\si$, $p$ and $\beta$, for any $0\leq s<t$ we  denote
\[\mathcal{E}_{s,t}:=C([s,t];\mathcal{B}^\si_{p}(D))\cap L^\beta(s,t;\mathcal{B}^\a_{p}(D)).\]
$\mathcal{E}_{s,t}$ turns out to be a Banach space, endowed with the norm
\[|v|_{\mathcal{E}_{s,t}}:=\sup_{r \in\,[s,t]}|v(r)|_{\mathcal{B}^\si_{p}(D)}+|v|_{L^p(s,t;\mathcal{B}^\a_{p}(D)}.\]
In the case $s=0$, we shall set $\mathcal{E}_{0,t}=\mathcal{E}_{t}$.

Our purpose here is to show that under condition \eqref{ca89} the family $\{u_\e\}_{\e \in\,(0,1)}$ satisfies a large deviation principle in $C([0,T];\mathcal{B}^\si_p(D))$, with action functional $I_T$, as defined in \eqref{action}. In view of Theorem \ref{teo-bd} and the arguments in Section \ref{sec3}, this follows once we prove that for any sequence $\{\varphi_\e\}_{\e>0}\subset \mathcal{A}^\gamma_T$, $\mathbb{P}$-almost surely convergent to some $\varphi \in\,\mathcal{A}^\gamma_T$, with respect to the topology of weak covergence in $L^2(0,T;H)$, the sequence $\{u^{\varphi_\e}_\e\}_{\e>0}$ converges $\mathbb{P}$-almost surely to $u^\varphi$ in $C([0,T];\mathcal{B}^\si_p(D))$.

\medskip

For any  $\e>0$, we introduce the random equation
\begin{equation}
\label{randombis}
\frac{dv_\e}{dt}(t)=A v_\e(t)+b(v_\e(t)+z_\e(t))+Q_\e\,\varphi_\e,\ \ \ \ v_\e(0)=u_0-z_\e(0),
\end{equation}
where
$z_\e(t)=z_\e^0(t)$ is the process introduced in \eqref{ca90}. In particular, we have
\[u^{\varphi_\e}_\e(t)-u^\varphi(t)=\le[v_\e(t)-u^\varphi(t)\r]+z_\e(t)=:\rho_\e(t)+z_\e(t),\ \ \ \ t\geq 0,\]
Since\[
\frac{d\rho_\e}{dt}(t)=A \rho_\e(t)+b(v_\e(t)+z_\e(t))-b(u^\varphi(t))+Q_\e\, \varphi_\e(t)-\varphi(t),\ \ \ \ \rho_\e(0)=-z_\e(0),
\]
we have that $\rho_\e(t)$ solves the following integral equation
\[\begin{array}{l}
\ds{\rho_\e(t)=-e^{tA}z_\e(0)+\int_0^te^{(t-s)A}\le(b(v_\e(s))-b(u^\varphi(s))\r)\,ds+\int_0^t e^{(t-s)A}b(z_\e(s))\,ds}\\
\vs
\ds{+\int_0^t e^{(t-s)A}\le(b(\rho_\e(s),z_\e(s))+b(z_\e(s),\rho_\e(s))\r)\,ds}\\
\vs
\ds{+\int_0^t e^{(t-s)A}\le(b(u^\varphi(s),z_\e(s))+b(z_\e(s),u^\varphi(s))\r)\,ds+\le[\Phi_\e(t)-\Gamma(\varphi)(t)\r]=:\sum_{i=1}^6 I_{\e,i}(t).}
\end{array}\]

Our first goal here is  to estimate the norm of each term $I_{\e,i}$ in $\mathcal{E}_{t}$, for every $t\leq T$, and prove a uniform bound for $\rho_\e$ in $\mathcal{E}_T$. To this purpose, we first prove a suitable bound for $u^\varphi$ in $H^\theta(D)$, with $\theta  \in\,(0,1)$.

\begin{Lemma}
Assume that $u_0 \in\,H^\theta(D)$, for some $\theta \in\,[0,1)$. Then, for any $\varphi \in\,L^2(0,T;H)$ we have
\begin{equation}
\label{ca96}
\sup_{t \in\,[0,T]}|u^\varphi(t)|^2_{H^\theta(D)}+\int_0^T|u^\varphi(s)|^2_{H^{\theta+1}(D)}\,ds\leq c\le(|u_{0}|_{H^\theta(D)},|\varphi|_{L^2(0,T;H)}\r).
\end{equation}

\end{Lemma}

\begin{proof}
Since
\[\frac 12\frac d{dt}|u^\varphi(t)|^2_{H}+|u^\varphi(t)|^2_{V}=\langle \varphi(t),u^\varphi(t)\rangle_H,\]
we immediately have 
\begin{equation}
\label{ca95}
|u^\varphi(t)|^2_{H}+\int_0^t|u^\varphi(s)|^2_{V}\,ds\leq |u_0|_H^2+\frac {\la_1}2\int_0^t|\varphi(s)|_H^2\,ds.\end{equation}

For every $\theta\geq 0$, we have
\[\frac 12\frac d{dt}|u^\varphi(t)|^2_{H^\theta(D)}+|u^\varphi(t)|^2_{H^{\theta+1}(D)}=\langle b(u^\varphi(t)),(-A)^{\theta}u^\varphi(t)\rangle_H+\langle \varphi(t),(-A)^{\theta}u^\varphi(t)\rangle_H.\]
Now, if we assume $\theta<1$ and set $q_1=2/(1-\theta)$ and $q_2=2/\theta$, we have
\[\le|\langle b(u^\varphi(t)),(-A)^{\theta}u^\varphi(t)\rangle_H\r|\leq |u^\varphi(t)|_{L^{q_1}(D)}|(-A)^\theta u^\varphi(t)|_{L^{q_2}(D)}|u^\varphi(t)|_V.\]
As
\[W^{\theta,2}(D)\hookrightarrow L^{q_1}(D),\ \ \ \ W^{1-\theta,2}(D)\hookrightarrow L^{q_2}(D),\]
this implies that 
\[\begin{array}{l}
\ds{\le|\langle b(u^\varphi(t)),(-A)^{\theta}u^\varphi(t)\rangle_H\r|\leq |u^\varphi(t)|_{H^\theta(D)}|u^\varphi(t)|_{H^{1+\theta}(D)}|u^\varphi(t)|_{V}}\\
\vs
\ds{\leq \frac 14|u^\varphi(t)|_{H^{1+\theta}(D)}^2+c\,|u^\varphi(t)|_{H^\theta(D)}^2|u^\varphi(t)|_{V}^2.}
\end{array}\]
Therefore, as
\[\le|\langle \varphi(t),(-A)^{\theta}u^\varphi(t)\rangle_H\r|\leq |\varphi(t)|_H\,|u^\varphi(t)|_{H^{2\theta}(D)}\leq \frac 14|u^\varphi(t)|_{H^{1+\theta}(D)}^2+c\,|\varphi|_H^2,\]
we conclude that 
\[\frac d{dt}|u^\varphi(t)|^2_{H^\theta(D)}+|u^\varphi(t)|^2_{H^{\theta+1}(D)}\leq c\,|u^\varphi(t)|_{H^\theta(D)}^2|u^\varphi(t)|_{V}^2+c\,|\varphi|_H^2.\]
Thanks to \eqref{ca95}, this implies 
\[\begin{array}{l}
\ds{|u^\varphi(t)|^2_{H^\theta(D)}\leq \exp\le(c\int_0^T|u^\varphi(s)|_V^2\,ds\r)\le(|u_0|_{H^{\theta}(D)}^2+c\,|\varphi|_{L^2(0,T;H)}^2\r)}\\
\vs
\ds{\leq \exp\le(c|u_0|_H^2+c\,|\varphi|_{L^2(0,T;H)}^2\r)\le(|u_0|_{H^{\theta}(D)}^2+c\,|\varphi|_{L^2(0,T;H)}^2\r),}
\end{array}\]
and \eqref{ca96} easily follows.
\end{proof}

Now, let us estimate each term $I_{\e,i}$, for $i=1,\ldots,6$.
Since
\[|I_{\e,1}(t)|_{\mathcal{E}_t}=\sup_{s \in\,[0,t]}\,|e^{sA}z_\e(0)|_{\mathcal{B}^\si_p(D)}+\le(\int_0^t |e^{sA}z_\e(0)|^\beta_{\mathcal{B}^\a_p(D)}\,ds\r)^{\frac 1\beta},\]
according to \eqref{ca103}, for any $t \leq T$ we have
\begin{equation}
\label{I_0}
\begin{array}{l}
\ds{|I_{\e,1}(t)|_{\mathcal{E}_t}\leq c\,|z_\e(0)|_{\mathcal{B}^\si_p(D)}+c\le(\int_0^t s^{-\frac 12(\a-\si)\beta}\,ds\r)^{\frac 1\beta}|z_\e(0)|_{\mathcal{B}^\si_p(D)}\leq c_T\,|z_\e(0)|_{\mathcal{B}^\si_p(D)}.}
\end{array}\end{equation}

Now, for any two processes $u(t)$ and $v(t)$, we define
\[\La(u,v)(t):=\int_0^t e^{(t-s)A}b(u(s),v(s))\,ds,\ \ \ \ t\geq 0.\]
By proceeding as in  \cite[proof of Lemma 6.3]{DPD02}, it is possible to show that 
if $v_1$ and $v_2$ are measurable mappings defined on $[0,T]$, with values in $\mathcal{B}^\a_p(D)$ and $ \mathcal{B}^\si_p(D)$,  respectively,   then 
\begin{equation}
\label{ca101}
|\La(v_i,v_j)(t)|_{\mathcal{B}^\si_p(D)}\leq c\int_0^t (t-s)^{-\frac 12(1+\frac 2p-\a)}|v_1(s)|_{\mathcal{B}^\a_p(D)}|v_2(s)|_{\mathcal{B}^\si_p(D)}\,ds,\ \ \ \ t\leq T,
\end{equation}
and
\begin{equation}
\label{ca102}
|\La(v_i,v_j)(t)|_{\mathcal{B}^\a_p(D)}\leq c\int_0^t (t-s)^{-\frac 12(1+\frac 2p-\si)}|v_1(s)|_{\mathcal{B}^\a_p(D)}|v_2(s)|_{\mathcal{B}^\si_p(D)}\,ds,\ \ \ \ t\leq T,
\end{equation}
both for $(i,j)=(1,2)$ and for $(i,j)=(2,1)$.

It is immediate to check that
\[b(v_\e(t))-b(u^\varphi(t))=b(\rho_\e(t))+b(\rho_\e(t),u^\varphi(t))+b(u^\varphi(t),\rho_\e(t)),\ \ \ t\geq 0,\]
so that, thanks to \eqref{ca103}, from \eqref{ca101} and \eqref{ca102}  we get 
\[|I_{\e,2}|_{\mathcal{E}_t}\leq c_1(t)|\rho_\e|_{L^\beta(0,t;\mathcal{B}^\a_p(D))}\,\le(|\rho_\e|_{C([0,t];\mathcal{B}^\si_p(D))}+|u^\varphi|_{C([0,t];\mathcal{B}^\si_p(D))}\r),\ \ \ t\geq 0,\]
for some continuous increasing function $c_1(t)$, such that $c_1(0)=0$.
Since we are assuming that  $\theta\geq \si+1-2/p$, we have that
$H^{\theta}(D)\hookrightarrow \mathcal{B}^\si_p(D)$,  so that from \eqref{ca96} we obtain\begin{equation}
\label{I_1}
|I_{\e,2}|_{\mathcal{E}_t}\leq c_1(t)\,c_{\gamma}(|u_0|_{H^\theta(D)})|\rho_\e|_{\mathcal{E}_t}\,\le(|\rho_\e|_{C([0,t];\mathcal{B}^\si_p(D))}+1\r).
\end{equation}

Concerning $I_{\e,3}(t)$, we first notice that 
\[b(z_\e(t))=\text{div}\,\le(z_\e(t)\otimes z_\e(t)\r)=\text{div}\,\le(z_\e(t)\otimes z_\e(t)-\e\,\vartheta_{\d(\e)} I\r),\ \ \ \ t\geq 0.\]
Then, since for every $\rho\geq -1$, $\eta\geq 0$ and $p\geq 2$ we have
\[|e^{tA}x|_{\mathcal{B}^\rho_p(D)}\leq c\,t^{-(1+\frac \rho 2-\frac 1p+\frac \eta 2)}|x|_{H^{-(1+\eta)}(D)},\ \ \ \ t>0,\]
from \eqref{ca147} we get
\[\begin{array}{l}
\ds{\le|I_{\e,3}(t)\r|_{\mathcal{B}^\si_p(D)}\leq c\int_0^t(t-s)^{-(1+\frac \si 2-\frac 1p+\frac \eta 2)}\le|\text{div}(z_\e(s)\otimes z_\e(s)-\e\,\vartheta_{\d(\e)}I)\r|_{[H^{-(1+\eta)}(D)]^4}\,ds}\\
\vs
\ds{\leq c\int_0^t(t-s)^{-(1+\frac \si 2-\frac 1p+\frac \eta 2)}\le|z_\e(s)\otimes z_\e(s)-\e\,\vartheta_{\d(\e)}I)\r|_{[H^{-\eta}(D)]^4}\,ds.}
\end{array}\]
In the same way, we have
\[\begin{array}{l}
\ds{\le|I_{\e,3}(t)\r|_{\mathcal{B}^\a_p(D)}\leq c\int_0^t(t-s)^{-(1+\frac \a 2-\frac 1p+\frac \eta 2)}\le|z_\e(s)\otimes z_\e(s)-\e\,\vartheta_{\d(\e)}I)\r|_{[H^{-\eta}(D)]^4}\,ds.}
\end{array}\]
Due to \eqref{ca103}, this implies that we can find $\eta >0$ and $\rho\geq 1$ such that
\begin{equation}
\label{ca104}
\le|I_{\e,3}\r|_{\mathcal{E}_t}\leq c_2(t)\,\le|z_\e\otimes z_\e-\e\,\vartheta_{\d(\e)}I)\r|_{L^\rho(0,T;[H^{-(1+\gamma)}(D)]^4)}.\end{equation}

For $I_{\e,4}(t)$, by using again \eqref{ca101} and \eqref{ca102}, we have
\[|I_{\e,4}(t)|_{\mathcal{B}^\si_p(D)}\leq c\int_0^t (t-s)^{-\frac 12(1+\frac 2p-\a)}|\rho_\e(s)|_{\mathcal{B}^\a_p(D)}|z_\e(s)|_{\mathcal{B}^\si_p(D)}\,ds,
\]
and
\[|I_{\e,4}(t)|_{\mathcal{B}^\a_p(D)}\leq c\int_0^t (t-s)^{-\frac 12(1+\frac 2p-\si)}|\rho_\e(s)|_{\mathcal{B}^\a_p(D)}|z_\e(s)|_{\mathcal{B}^\si_p(D)}\,ds,
\]
and then, according to \eqref{ca103}, we can find $\rho\geq 1$ such that
\begin{equation}
\label{I_3}
|I_{\e,4}|_{\mathcal{E}_t}\leq c_3(t)\,|\rho_\e|_{L^\beta(0,t;\mathcal{B}^\a_p(D))}|z_\e|_{L^\rho(0,T;\mathcal{B}^\si_p(D))}\leq c_3(t)\,|\rho_\e|_{\mathcal{E}_t}|z_\e|_{L^\rho(0,T;\mathcal{B}^\si_p(D))}.
\end{equation}

As for $I_{\e,4}(t)$, for $I_{\e,5}(t)$ we have
\[|I_{\e,5}(t)|_{\mathcal{B}^\si_p(D)}\leq c\int_0^t (t-s)^{-\frac 12(1+\frac 2p-\a)}|u^\varphi(s)|_{\mathcal{B}^\a_p(D)}|z_\e(s)|_{\mathcal{B}^\si_p(D)}\,ds,
\]
and
\[|I_{\e,5}(t)|_{\mathcal{B}^\a_p(D)}\leq c\int_0^t (t-s)^{-\frac 12(1+\frac 2p-\si)}|u^\varphi(s)|_{\mathcal{B}^\a_p(D)}|z_\e(s)|_{\mathcal{B}^\si_p(D)}\,ds.
\]
As we are assuming      $\theta \geq \si+1-2/p$, we have that $\theta>\a-2/p$, so that for any $\eta>0$ such that 
$\theta-\eta >\a-2/p$, we have $H^{1+\theta-\eta}(D)\hookrightarrow \mathcal{B}_p^\a(D)$. By interpolation, this implies
\[|x|_{\mathcal{B}^\a_p(D)}\leq c_\eta\, |x|_{H^{1+\theta-\eta}(D)}\leq c_\eta\,|x|_{H^{1+\theta}(D)}^{1-\eta}|x|_{H^{\theta}(D)}^{\eta},\]
so that 
\[|x|_{\mathcal{B}^\a_p(D)}^{\frac 2{1-\eta}}\leq c_\eta\,|x|_{H^{1+\theta}(D)}^{2}|x|_{H^{\theta}(D)}^{\frac{2\eta}{1-\eta}}.\]
According to \eqref{ca96}, this implies that $u^\varphi \in\,L^{\frac 2{1-\eta}}(0,T;\mathcal{B}^\a_p(D))$ and
\begin{equation}
\label{ca105}
|u^\varphi|_{L^{\frac 2{1-\eta}}(0,T;\mathcal{B}^\a_p(D))}\leq c_{\gamma,\eta}(|u_0|_{H^\theta(D)}).
\end{equation}

Due to condition \eqref{ca103}, since $\theta\geq \si+1-2/p$, we can find $\eta \in\,(0,1)$ such that
\[1-\frac 2\beta<\eta<\theta+\frac 2p-\a.\]
For such $\eta>0$ we have
\begin{equation}
\label{ca110}
\begin{array}{l}
\ds{|I_{\e,5}(t)|_{\mathcal{B}^\si_p(D)}\leq c\,\le(\int_0^t s^{-\frac 12(1+\frac 2p-\a)\frac \beta{\beta-1}}\,ds\r)^{\frac{\beta-1}\beta}\,|u^\varphi|_{L^{\frac 2{1-\eta}}(0,t;\mathcal{B}^\a_p(D))}|z_\e|_{L^{\kappa}(0,t;\mathcal{B}^\si_p(D))},}
\end{array}
\end{equation}
where
\[\frac 1\kappa=1-\le[\frac{1-\eta}2+\frac{\beta-1}\beta\r]=\frac{1}\beta-\frac{1-\eta}2.\]
Analogously, if we pick $\eta>1-2/p$, we get
\[\int_0^t|I_{\e,5}(s)|^\beta_{\mathcal{B}^\a_p(D)}\,ds\leq c\,\le(\int_0^t s^{-\frac 12(1+\frac 2p-\si)}\,ds\r)^\beta\,|u^\varphi|_{L^{\frac 2{1-\eta}}(0,t;\mathcal{B}^\a_p(D))}^\beta|z_\e|_{L^{\frac{2\beta}{2-\beta(1-\eta)}}(0,t;\mathcal{B}^\si_p(D))}.\]
Thanks to \eqref{ca105}, this, together with \eqref{ca110}, implies that there exists some $\rho\geq 1$ such that
\begin{equation}
\label{I_4}
|I_{\e,5}|_{\mathcal{E}_T}\leq c_4(t)\,c_{\gamma}(|u_0|_{H^\theta(D)})|z_\e|_{L^{\rho}(0,t;\mathcal{B}^\si_p(D))}.
\end{equation}

Collecting together \eqref{I_0}, \eqref{I_1}, \eqref{ca104}, \eqref{I_3} and \eqref{I_4}, we conclude that
\[\begin{array}{l}
\ds{|\rho_\e|_{\mathcal{E}_t}\leq c(t)\,c_{\gamma}(|u_0|_{H^\theta(D)})|\rho_\e|_{\mathcal{E}_t}\,\le(|\rho_\e|_{C([0,t];\mathcal{B}^\si_p(D))}+|z_\e|_{L^\rho(0,T;\mathcal{B}^\si_p(D))}+1\r)+c_T\,|z_\e(0)|_{\mathcal{B}^\si_p(D))}}\\
\vs
\ds{+c\,(t)\,c_{\gamma}(|u_0|_{H^\theta(D)})\le(|z_\e|_{L^{\rho}(0,T;\mathcal{B}^\si_p(D))}+\le|z_\e\otimes z_\e-\e\,\vartheta_{\d(\e)}I\r|_{L^\rho(0,T;[H^{-\gamma}(D)]^4)}\r)+\le|\Phi_\e-\Gamma(\varphi)\r|_{\mathcal{E}_T},}
\end{array}\]
for some continuous increasing function $c(t)$ such that $c(0)=0$.

Now, we are going to show that for any sequence $\{\e_n\}_{n \in\,\nat}$ converging to zero, there exists a subsequence $\{\e_{n_k}\}_{k \in\,\nat}\subset \{\e_n\}_{n \in\,\nat}$, such that
\begin{equation}
\label{ca146}
\lim_{k\to\infty}|\rho_{\e_{n_k}}|_{\mathcal{E}_T}=0,\ \ \ \ \mathbb{P}-\text{a.s.}\end{equation}
and this clearly implies that
\[\lim_{\e\to 0}|\rho_{\e}|_{\mathcal{E}_T}=0,\ \ \ \ \mathbb{P}-\text{a.s.}
\]
As $u^{\varphi_\e}_\e(t)-u^\varphi(t)=\rho_\e(t)+z_\e(t)$, for $t \in\,[0,T]$, according to \eqref{ca115} we can conclude that 
\begin{equation}
\label{ca145}
\lim_{\e\to 0}\sup_{t \in\,[0,T]}|u^{\varphi_\e}_\e(t)-u^\varphi(t)|_{\mathcal{B}_p^\si(D)}=0,\ \ \ \ \mathbb{P}-\text{a.s.}
\end{equation}

Let $\{\e_n\}_{n \in\,\nat}$ be a sequence converging to zero.  As we are assuming that $\a<2/p$, there exists $\rho<1$ such that $H^\rho(D)\hookrightarrow \mathcal{B}^\a_p(D)$, so that, due to \eqref{ca30} we have that
\begin{equation}
\label{ca121}
\lim_{\e\to 0} \le|\Phi_\e-\Gamma(\varphi)\r|_{\mathcal{E}_T}=0,\ \ \ \ \mathbb{P}-\text{a.s}.\end{equation}
Then, as a consequence of \eqref{ca115}, \eqref{renormalization} and \eqref{ca121}, we have that there exists a subsequence of $\{\e_n\}_{n \in\,\nat}$, that for simplicity of notations we are still denoting by $\{\e_n\}_{n \in\,\nat}$, and a set $\Omega^\prime \subseteq \Omega$ with $\mathbb{P}(\Omega^\prime)=1$, such that
\begin{equation}
\label{ca150}
\begin{array}{ll}
\ds{\lim_{n\to \infty}}   &  \ds{\le(|z_{\e_n}(\omega)|_{C([0,T];\mathcal{B}^\si_p(D))}+
\le|z_\e(\omega)\otimes z_\e(\omega)-\e\,\vartheta_{\d(\e)}I\r|_{L^\rho(0,T;[H^{-\gamma}(D)]^4)}\r.}\\
&  \vs
&  \ds{\le. +\le|\Phi_{\e_n}(\omega)-\Gamma(\varphi)(\omega)\r|_{\mathcal{E}_T}\r)=0,\ \ \ \ \omega \in\,\Omega^\prime.}
\end{array}\end{equation}

Next, for any $\e>0$  we  denote
\[\tau_\e:=\inf\,\le\{\,t\geq 0\ :\ |\rho_\e(t)|_{\mathcal{B}^\si_{p}(D)}\geq 1\,\r\}.\]
If we fix any $\omega \in\,\Omega^\prime$, in view of \eqref{ca115} there exists some $n_0=n_0(\omega) \in\,\nat$ such that for any $n\geq n_0$ and $t\leq \tau_{\e_n}(\omega)$
\[\begin{array}{l}
\ds{|\rho_\e(\omega)|_{\mathcal{E}_t}\leq 3\,c(t)\,c_{\gamma}(|u_0|_{H^\theta(D)})|\rho_\e(\omega)|_{\mathcal{E}_t}\,+c_T\,|z_\e(0)|_{\mathcal{B}^\si_p(D))}+\le|\Phi_\e(\omega)-\Gamma(\varphi)(\omega)\r|_{\mathcal{E}_T}}\\
\vs
\ds{+c\,(t)\,c_{\gamma}(|u_0|_{H^\theta(D)})\le(|z_\e(\omega)|_{L^{\rho}(0,T;\mathcal{B}^\si_p(D))}+\le|z_\e(\omega)\otimes z_\e(\omega)-\e\,\vartheta_{\d(\e)}I\r|_{L^\rho(0,T;[H^{-\gamma}(D)]^4)}\r),}

\end{array}\]
This implies that if we take $t_0>0$ such that
\[3\,c(t_0)\,c_{\gamma}(|u_0|_{H^\theta(D)})\leq \frac 12,\]
for any $n\geq n_0$ and $t\leq \tau_{\e_n}(\omega)\wedge t_0$
\[\begin{array}{l}
\ds{|\rho_{\e_n}(\omega)|_{\mathcal{E}_t}\leq c\,\le|\Phi_{\e_n}(\omega)-\Gamma(\varphi)(\omega)\r|_{\mathcal{E}_T}}\\
\vs
\ds{+c_T\le(|z_{\e_n}(\omega)|_{C([0,T];\mathcal{B}^\si_p(D))}+\le|z_{\e_n}(\omega)\otimes z_{\e_n}(\omega)-{\e_n}\,\vartheta_{\d(\e)}I\r|_{L^\rho(0,T;[H^{-\gamma}(D)]^4)}\r).}
\end{array}\]

As a consequence of \eqref{ca150}, there exists $n_1=n_1(\omega)\geq n_0$ such that 
\[\begin{array}{l}
\ds{c_T\le(|z_{\e_n}(\omega)|_{C([0,T];\mathcal{B}^\si_p(D))}+\le|z_{\e_n}(\omega)\otimes z_{\e_n}(\omega)-{\e_n}\,\vartheta_{\d(\e)}I\r|_{L^\rho(0,T;[H^{-\gamma}(D)]^4)}\r)}\\
\vs
\ds{+c\,\le|\Phi_{\e_n}(\omega)-\Gamma(\varphi)(\omega)\r|_{\mathcal{E}_T}\leq \frac 12,\ \ \ n\geq n_1,}
\end{array}\]
so that $\tau_{\e_n}(\omega)\wedge t_0=t_0$, for $n\geq n_1$, and
\[\begin{array}{l}
\ds{|\rho_{\e_n}(\omega)|_{\mathcal{E}_{t_0}}\leq c\,\le|\Phi_{\e_n}(\omega)-\Gamma(\varphi)(\omega)\r|_{\mathcal{E}_T}}\\
\vs
\ds{+c_T\le(|z_{\e_n}(\omega)|_{C([0,T];\mathcal{B}^\si_p(D))}+\le|z_{\e_n}(\omega)\otimes z_{\e_n}(\omega)-{\e_n}\,\vartheta_{\d(\e)}I\r|_{L^\rho(0,T;[H^{-\gamma}(D)]^4)}\r).}
\end{array}\]

Now, we can repeat the same argument in the intervals $[(i-1)t_0,it_0]$, for $i=0,\ldots,i_T$, where $i_T$ is the smallest integer such that $i_T t_0\geq T$, and we find
\begin{equation}
\label{ca120}
\begin{array}{l}
\ds{|\rho_{\e_n}(\omega)|_{\mathcal{E}_{(i-1)t_0,i t_0}}\leq i\,c\,\le|\Phi_{\e_n}(\omega)-\Gamma(\varphi)(\omega)\r|_{\mathcal{E}_T}}\\
\vs
\ds{+i\,c_T \le(|z_{\e_n}(\omega)|_{C([0,T];\mathcal{B}^\si_p(D))}+\le|z_{\e_n}(\omega)\otimes z_{\e_n}(\omega)-{\e_n}\,\vartheta_{\d(\e)}I\r|_{L^\rho(0,T;[H^{-\gamma}(D)]^4)}\r),}
\end{array}
\end{equation}
for every $n\geq n_i=n_i(\omega)$, where $n_i(\omega)\geq  n_{i-1}(\omega)$ is such that
\[\begin{array}{l}
\ds{c_T \le(|z_{\e_n}(\omega)|_{C([0,T];\mathcal{B}^\si_p(D))}+\le|z_{\e_n}(\omega)\otimes z_{\e_n}(\omega)-{\e_n}\,\vartheta_{\d(\e)}I\r|_{L^\rho(0,T;[H^{-\gamma}(D)]^4)}\r)}\\
\vs
\ds{+c\,\le|\Phi_{\e_n}(\omega)-\Gamma(\varphi)(\omega)\r|_{\mathcal{E}_T}   \leq \frac 1{2i},\ \ \ \ n\geq n_i.}
\end{array}\]   

Therefore, from \eqref{ca120} we obtain that for any $n\geq n_{i_T}(\omega)$
\[\begin{array}{l}
\ds{|\rho_{\e_n}(\omega)|_{\mathcal{E}_{T}}\leq i\,c\,\le|\Phi_{\e_n}(\omega)-\Gamma(\varphi)(\omega)\r|_{\mathcal{E}_T}}\\
\vs
\ds{+i\,c_T \le(|z_{\e_n}(\omega)|_{C([0,T];\mathcal{B}^\si_p(D))}+\le|z_{\e_n}(\omega)\otimes z_{\e_n}(\omega)-{\e_n}\,\vartheta_{\d(\e)}I\r|_{L^\rho(0,T;[H^{-\gamma}(D)]^4)}\r),}
\end{array}\]
and due to \eqref{ca150} 
we  can conclude that
\[\lim_{n\to \infty}|\rho_{\e_n}(\omega)|_{\mathcal{E}_T}=0.\]
           
\appendix

\section{Appendix}
Here we describe and prove some properties of the solution of the linear problem. As in Section \ref{sec2},  for every $\a\geq 0$ and $\e>0$ we denote by $z^\a_\e(t)$ the solution of the linear problem 
\[
dz(t)=(A-\a)z(t)\,dt+\sqrt{\e}\,dw^{\d(\e)}(t),\ \ \ \ t\geq 0.
\]
The process $z^\a_\e(t)$  is given by
\[z_{\e}^\a(t)=\sqrt{\e}\int_{-\infty}^t e^{(t-s)(A-\a)}\,d\bar{w}^{\d(\e)}(s),\ \ \ \ t \geq 0.
\]
As we already mentioned in Section \ref{sec2},  for any  fixed $\e>0$ the process $z^\a_\e$ belongs to the space $L^p(\Omega;C([0,T];D((-A)^{\beta}))$, for any $T>0$, $p\geq 1$ and $\beta<\gamma/2$. 

We first want to estimate the norm of $z^\a_\e$ in Besov spaces of negative exponent. 

\begin{Lemma}
\label{lemma-A1}
For any $\a\geq 0$ and $\e>0$ and for any $p, \kappa\geq 1$ and $\si<\si^\prime<0$ it holds
\begin{equation}
\label{ca115}
\E\,\sup_{t \in\,[0,T]}|z^\a_\e(t)|^\kappa_{\mathcal{B}^\si_p(D)}\leq c_{\kappa,p}\,\le(\,\e\sum_{k \in\,\mathbb{Z}^2_0}|k|^{2(\si^\prime-1)}\r)^{\frac \kappa 2}.\end{equation}
\end{Lemma}

\begin{proof}
Since $z^\a_\e(t)=(-A)^{-\frac \si 2}(-A)^{\frac\si 2}z^\a_\e(t)$, we have
\begin{equation}
\label{ca142}
|z^\a_\e(t)|_{\mathcal{B}_p^\si(D)}\leq |(-A)^{\frac\si 2}z^\a_\e(t)|_{L^p(D)}.\end{equation}
By using stochastic factorization, for any $\beta \in\,(0,1)$ we have
\[(-A)^{\frac\si 2}z^\a_\e(t)=\frac{\sin \pi\beta}{\pi}\int_{-\infty}^t(t-s)^{\beta-1}e^{(t-s)A}Y_{\e,\beta}(s)\,ds,\]
where
\[Y_{\e,\beta}(s)=\int_{-\infty}^s (s-\rho)^{-\beta}e^{(s-\rho)A} (-A)^{\frac \si 2}dw^{\d(\e)}(\rho).\]
Therefore, 
if we take $p\geq 1/\beta$, we get
\begin{equation}
\label{ca141}
|(-A)^{\frac\si 2}z^\a_\e(t)|^p_{L^p(D)}\leq c_{\beta,p}\le(\int_{-\infty}^t s^{-\frac{(1-\beta)p}{p-1}}e^{-\frac{p}{p-1}s}\,ds\r)^{p-1}\int_{-\infty}^t|Y_{\e,\beta}(s)|_{L^p(D)}^p\,ds.\end{equation}
Now, for any $t \in\,\reals$ and $x \in\,D$
\[\begin{array}{l}
\ds{\E\,|Y_{\e,\beta}(t,x)|^p= c_p\,\e^{\frac p 2}\E\,\le(\,\sum_{k \in\,\mathbb{Z}^2_0}\int_{-\infty}^t\la_{k}(\d(\e)) |k|^{\si}(t-s)^{-\beta}\,e^{-(t-s)(|k|^2+\a)}e_k(x)\,d\beta_k(s)\r)^p}\\
\vs
\ds{\leq c_p\,\e^{\frac p2}\le(\,\sum_{k \in\,\mathbb{Z}^2_0}\int_{-\infty}^t\la_{k}(\d(\e))^2 |k|^{2\si}(t-s)^{-2\beta}\,e^{-2(t-s)(|k|^2+\a)}|e_k(x)|^2\,ds\r)^{\frac p2}}\\
\vs
\ds{\leq c_p\,\e^{\frac p2}\le(\,\sum_{k \in\,\mathbb{Z}^2_0}|k|^{2\si+4\beta-2}\r)^{\frac p 2},}
\end{array}\]
so that, integrating  with respect to $x \in\,D$, for any $\beta<-\si/2$, and hence $p\geq -2/\si$, 
\[\E\,|Y_{\e,\beta}(t)|^p_{L^p(D)} \leq c_{ p}\,\le(\,\e \sum_{k \in\,\mathbb{Z}^2_0}|k|^{2(\si+2\beta-1)}\r)^{\frac p 2}.\]
Therefore, thanks to \eqref{ca142} and \eqref{ca141}, for any $\kappa\geq p\geq 2/\si$ this yields
\[\begin{array}{l}
\ds{\E\,\sup_{t \in\,[0,T]}|z^\a_\e(t)|^\kappa_{\mathcal{B}^\si_p(D)}\leq \E\,\sup_{t \in\,[0,T]}|(-A)^{\frac\si 2}z^\a_\e(t)|_{L^p(D)}^\kappa\leq c_{\kappa,p}\, \E\,\sup_{t \in\,[0,T]}|(-A)^{\frac\si 2}z^\a_\e(t)|_{L^k(D)}^\kappa}\\
\vs
\ds{\leq c_{\kappa, p}\,\le(\,\e \sum_{k \in\,\mathbb{Z}^2_0}|k|^{2(\si+2\beta-1)}\r)^{\frac \kappa 2}.}
\end{array}\]
The general case follows from the H\"older inequality.
\end{proof}

Next, we estimate the norm of $z^\a_\e$ in $L^p(D)$-spaces.

\begin{Lemma}
For every $\a\geq 0$ and $\e>0$ and for every $p\geq 1$ it holds
\begin{equation}
\label{ca50}
\E\,|z^\a_\e(t)|_{L^p(D)}^p\leq c_p(T)\,\le(\e\,\log \le(\frac {1+\d(\e)}{\d(\e)}\r)\r)^{\frac p2},\ \ \ \ t \in\,[0,T].
\end{equation}
\end{Lemma}

\begin{proof}
 For every $p\geq 1$ we have
\[\begin{array}{l}
\ds{\E\,|z^\a_\e(t)|_{L^p(D)}^p=\e^{\frac p2}\,\E\int_D\le|\sum_{k \in\,\mathbb{Z}^2_0}\int_{-\infty}^te^{-(t-s)(|k|^2+\a)}\la_{k}(\d(\e)) e_k(x)\,d\bar{\beta}_k(s)\r|^p\,dx}\\
\vs
\ds{\leq \e^{\frac p2}\,\int_D\le(\,\sum_{k \in\,\mathbb{Z}^2_0} e^{-2(t-s)(|k|^2+\a)}\la_{k}(\d(\e))^2\,  |e_k(x)|^2\,ds\r)^{\frac p2}\,dx}\\
\vs
\ds{\leq |D|\,\e^{\frac p2}\le(\,\sum_{k \in\,\mathbb{Z}^2_0}\frac{1}{|k|^2(1+\d(\e)|k|^{2\gamma})}\r)^{\frac p2}.}
\end{array}\]
Since we have 
\[\begin{array}{l}
\ds{\sum_{k \in\,\mathbb{Z}^2_0}\frac{1}{|k|^2(1+\d(\e)|k|^{2\gamma})}\sim \int_1^{+\infty}\frac 1{x(1+\d(\e)x^\gamma)}\,dx}\\
\vs
\ds{=\frac 1\gamma\int_{\d(\e)}^\infty \frac 1{x(1+x)}\,dx=\frac 1\gamma\le(\log(1+\d(\e))+\log \frac 1{\d(\e)}\r),}
\end{array}\]
this implies that \eqref{ca50} holds.
\end{proof}

 Now, by proceeding as in the proof of \cite[Proposition 2.1]{DPD07}, it is possible to show that for any $p\geq 1$ and $\beta \in\,(0,1/4)$ there exist $\theta=\theta(p,\beta) \in\,(0,1/4)$ and $\rho=\rho(p,\beta) \in\,(0,1)$, and a random variable $K_\e(p,\beta)$ such that for any $\a\geq 0$ and $\e>0$
 \begin{equation}
 \label{sa1}
 |z_\e^\a(t)|_{L^p(D)}\leq (\a\vee 1)^{-\theta}(1+t^\rho)\,K_{\e}(p,\beta),\ \ \ \ \mathbb{P}-\text{a.s.},
 \end{equation}
 where
 \begin{equation}
 \label{kappa1}
 K_\e(p,\beta)=c_{p,\beta}\le(\int_{-\infty}^{+\infty}(1+\si^2)^{-1}|Y_\e(\si)|_{L^p(D)}^{m}\,d\si\r)^{1/m},\end{equation}
for some $m=m(p,\beta)\geq 1$,  and where
 \begin{equation}
 \label{kappa2}
 Y_\e(\si)=\sqrt{\e}\int_{-\infty}^\si (\si-s)^{-\beta}e^{(\si-s)A}\,dw^{\d(\e)}(s).
 \end{equation}
In particular,  we have
 \begin{equation}
 \label{ca40}
 |z_\e^\a|_{C([0,T];L^p(D))}\leq (\a\vee 1)^{-\theta}\,c_{p}(T)\, K_{\e}(p,\beta),\ \ \ \ \mathbb{P}-\text{a.s.}
 \end{equation}

In what follows, it will be important that the random variable $K_\e(p,\beta)$ has all moments finite, with an uniform bound with respect to $\e>0$. 

\begin{Lemma}
Let $p, q\geq 1$ and $\e>0$ be fixed. Then, for any $\eta \in\,(0,1/2\gamma)$ there exists $\beta_\eta \in\,(0,1/4)$ such that
 \begin{equation}
 \label{sa2}
\E \,|K_\e(p,\beta_\eta)|^q\leq c_{p,\beta_\eta,q}\le(\e\,\d(\e)^{-\eta}\r)^{c_{q,p}},\end{equation}

\end{Lemma}

\begin{proof} It is immediate to check that, for any $q\geq m$, we have
\[\E |K_{\e}(p,\beta)|^q \leq c_{p,\beta,q}\int_{-\infty}^{+\infty}(1+\si^2)^{-1}\E\,|Y_\e(\si)|_{L^p(D)}^{q}\,d\si.\]
 Now, since
 \[Y_\e(\si,x)=\sqrt{\e}\sum_{k \in\,\mathbb{Z}^2_0}\int_{-\infty}^\si (\si-s)^{-\beta}\la_k(\d(\e))e^{-|k|^2(\si-s)}e_k(x)\,d\bar{\beta}_k(s),\]
we have
\[\begin{array}{l}
\ds{\E\,\le| Y_\e(\si,x)\r|^p\leq c_p\,\e^{p/2}\le(\sum_{k \in\,\mathbb{Z}^2_0}|e_k|^2_{L^\infty(D)}\int_0^\infty s^{-2\beta}  \la_k(\d(\e))^2e^{-|k|^2s}\,ds\r)^{\frac p2}}\\
\vs
\ds{\leq c_p\le(\e\sum_{k \in\,\mathbb{Z}^2_0}|k|^{-2(1-2\beta)}(1+\d(\e)|k|^{2\gamma})^{-1}\r)^{\frac p2}=:c_p\,\La_\beta(\e)^{\frac p2}.}
\end{array}\]
This  implies that for any $p, q\geq 1$
\[\E\,|Y_\e(\si)|_{L^p(D)}^q\leq c_{1}(q,p)\,\La_\beta(\e)^{c_{2}(q,p)},\]
for some positive constants $c_{1}(q,p)$ and $c_{2}(q,p)$.
Now, we have
\[\begin{array}{l}
\ds{\La_\beta(\e)\sim \e\int_1^{+\infty}\frac 1{x^{1-2\beta}(1+\d(\e)x^\gamma)}\,dx=\e\,\frac 1\gamma\le(\frac 1{\d(\e)}\r)^{\frac {2\beta} \gamma}\int_{\d(\e)}^{+\infty}y^{\frac{2\beta}\gamma-1}\frac 1{1+y}\,dy.}
\end{array}\]
Therefore, if we pick any $\eta \in\,(0,1/2\gamma)$ and define $\beta_\eta: =\eta\,\gamma/2$, we have
\[\La_{\beta_\eta}(\e)\leq c\,\e\,\d(\e)^{-\eta},\]
and this implies \eqref{sa2}.
 
\end{proof}

\bigskip

In what follows, we shall denote  $\mathcal{H}:=\mathbb{R}^{\mathbb{Z}^2_0}$ and $\mu:=\mathcal{N}(0,(-A)^{-1}/2).$
The Gaussian measure $\mu$ is defined on $\mathcal{H}$, but in fact $\mu(H^\si(D))=1$, if $\si<0$, so that the support of $\mu$ is contained in  $H^\si(D)$, for every $\si<0$.

Now, for any $h \in\,\mathcal{H}$ and $\d>0$, we define
\[h_\d:=\sum_{k \in\,\mathbb{Z}^2_0} \langle h,e_k\rangle \la_{k}(\d)\, e_k,\]
where we recall that, for any $k \in\,\mathbb{Z}^2_0$ and $\d>0$,
\[\la_{k}(\d)=\frac 1{\sqrt{1+\d\,|k|^{2\gamma}}}.\]
Next, for $i=1,2$ we define 
\begin{equation}
\label{ca131}
:(h^i_\d)^2:(x)=\sqrt{2}\,\le[(h^i_\d)^2(x)-\vartheta_\d\r],\ \ \ x \in\,D,\ \ \d>0,\end{equation}
where
\[\vartheta_\d=\frac 1{2(2\pi)^2}\sum_{k \in\,\mathbb{Z}^2_0}\frac{k_1^2}{|k|^4}\la_{k}(\d)^2=\frac 1{2(2\pi)^2}\sum_{k \in\,\mathbb{Z}^2_0}\frac{k_2^2}{|k|^4}\la_{k}(\d)^2.\]
By proceeding as in \cite[Appendix]{DPD02} it is possible to prove that for $i=1,2$
\[
\exists  \lim_{\d\to 0}:(h^i_\d)^2:\ \ \ \ \text{in}\ L^\kappa(\mathcal{H},\mu;H^{\si}(D)),
\]
and
\[
\exists  \lim_{\d\to 0}h^1_\d\,h^2_\d\ \ \ \ \text{in}\ L^\kappa(\mathcal{H},\mu;H^{\si}(D)),
\]
for every $\kappa\geq 1$ and $\si<0$. In particular, due to definition \eqref{ca131}, this implies that
\begin{equation}
\label{ca130}
\exists\lim_{\d\to 0}\le(h_\d\otimes h_\d-\vartheta_\d\,I_{\mathbb{R}^2}\r),\ \ \ \ \text{in}\ L^\kappa(\mathcal{H},\mu;[H^{\si}(D)]^4).
\end{equation}

\begin{Lemma}
For every $\e>0$,  let us denote $z_\e(t):=z^0_\e(t)$. Then, for $\si<0$ and $\kappa, p\geq 1$ we have
\begin{equation}
\label{renormalization}
\lim_{\e\to 0}\,\E\,|z_\e\otimes z_\e-\e\,\vartheta_{\d(\e)}\,I|^\kappa_{L^p(0,T;[H^\si(D)]^4)}=0.\end{equation}
\end{Lemma}

\begin{proof}
It is immediate to check that
\[z_\e(t)=\sqrt{\e}\,Q_\e z(t),\ \ \ \ t \in\,\mathbb{R},\]
where
\[z(t)=\int_{-\infty}^t e^{(t-s)A}dw(t)=\sum_{k \in\,\mathbb{Z}^2_0}\int_{-\infty}^t e^{-(t-s)|k|^2}\,d\beta_k(s).\]
The process $z(t)$ is  stationary Gaussian  and 
$\mathcal{L}(z(t))=\mu$, for every $t \in\,\mathbb{R}.$ This means that for any $p\geq 1$
\[\begin{array}{l}
\ds{\E\,|z_\e\otimes z_\e-\e\,\vartheta_{\d(\e)}\,I|^p_{L^p(0,T;[H^\si(D)]^4)}=\E\int_0^T|z_\e(t)\otimes z_\e(t)-\e\,\vartheta_{\d(\e)}\,I|^p_{[H^\si(D)]^4}\,dt}\\
\vs
\ds{=\e^p\,T\,\int_{\mathcal{H}}|Q_\e h\otimes Q_\e h-\vartheta_{\d(\e)}I|^p_{[H^\si(D)]^4}\mu(dh)=\e^p\,T\,\int_{\mathcal{H}}|h_{\d(\e)}\otimes  h_{\d(\e)}-\vartheta_{\d(\e)}I|^p_{[H^\si(D)]^4}\,\mu(dh).
}
\end{array}\]
Because of \eqref{ca130}, this implies \eqref{renormalization} in the case $\kappa=p\geq 1$. The case $\kappa, p\geq 1$ follows from the H\"older inequality and the fact that $L^p(D)\subset L^q(D)$, if $p\geq q$.

\end{proof}

\end{document}